\documentclass[reqno,12pt]{amsart}
\usepackage{color}
\usepackage[margin=1in]{geometry}
\usepackage{amsmath,amssymb}
\usepackage{amsthm}
\usepackage{paralist} 
\usepackage{graphicx}  
\usepackage{bbm}
\usepackage{subcaption}
\usepackage{fancyhdr}
\usepackage{indentfirst}
\usepackage{booktabs}
\usepackage{verbatim}
\usepackage{tcolorbox}
\usepackage{mathtools}
\usepackage{bm}
\usepackage{wasysym}
\usepackage{empheq}
\usepackage{tabularx}
\graphicspath{{figs_TSP/}}
\usepackage[font=small,labelfont=bf]{caption} 
\usepackage{paralist}
\usepackage{algorithm}
\usepackage{algpseudocode}
\usepackage{hyperref}

\newtheorem{theorem}{Theorem}[section]

\newtheorem{lemma}[theorem]{Lemma}
\newtheorem{proposition}[theorem]{Proposition}

\newtheorem{assumption}{Assumption}

\newtheorem{remark}{Remark}

\newcommand{\E}{\mathbb{E}}

\newcommand{\R}{\mathbb{R}}
\renewcommand{\P}{\mathbb{P}}

\newcommand{\N}{\mathbb{N}}

\newcommand{\CA}{\mathcal{A}}
\newcommand{\CB}{\mathcal{B}}
\newcommand{\CE}{\mathcal{E}}

\newcommand{\CL}{\mathcal{L}}

\newcommand{\CM}{\mathcal{M}}

\newcommand{\CT}{\mathcal{T}}
\newcommand{\CW}{\mathcal{W}}
\newcommand{\CP}{\mathcal{P}}

\newcommand{\CV}{\mathcal{V}}

\newcommand{\argmin}{\operatornamewithlimits{argmin}}

\def\P{\mathbb{P}}

\def\t{\tilde}

\def\R{\mathbb{R}}
\def\E{\mathbb{E}}

\def\P{\mathbb{P}}

\def\t_i{{t_i}}

\def\t{\tilde}
\newcommand{\HZnew}[1]{#1}
\newcommand{\YMnew}[1]{#1}

\title[Consensus-Based Method for Composite Optimization]{ProxiCBO: A Provably Convergent Consensus-Based Method for Composite Optimization}

\author[H. Zhang]{Haoyu Zhang\textsuperscript{1}}
\thanks{\textsuperscript{1} Department of Mathematics, University of California San Diego, San Diego, CA 92093 USA. 
(haz053@ucsd.edu)}
\address{(HZ) Department of Mathematics, University of California San Diego, San Diego, CA 92093 USA.}
\email{haz053@ucsd.edu}

\author[Y. Ma]{Yanting Ma\textsuperscript{2} \textsuperscript{$\dagger$}}
\thanks{\textsuperscript{2}Mitsubishi Electric Research Laboratories (MERL), Cambridge, MA 02139 USA.
(yma@merl.com)}
\address{(YM) Mitsubishi Electric Research Laboratories (MERL), Cambridge, MA 02139 USA.} \email{yma@merl.com}

\author[R. Kitichotkul]{Ruangrawee Kitichotkul\textsuperscript{3}}
\thanks{\textsuperscript{3} Department of Electrical and Computer Engineering, Boston University, Boston, MA 02215 USA.
(rkitich@bu.edu)}
\address{(RK) Department of Electrical and Computer Engineering, Boston University, Boston, MA 02215 USA.} \email{rkitich@bu.edu}

\author[J. Rapp]{Joshua Rapp\textsuperscript{4}}
\thanks{\textsuperscript{4}  Mitsubishi Electric Research Laboratories (MERL), Cambridge, MA 02139 USA.
(rapp@merl.com)} 
\address{(JR)  Mitsubishi Electric Research Laboratories (MERL), Cambridge, MA 02139 USA.}
\email{rapp@merl.com}

\author[P. T. Boufounos]{Petros T. Boufounos\textsuperscript{5}}
\thanks{\textsuperscript{5}  Mitsubishi Electric Research Laboratories (MERL), Cambridge, MA 02139 USA.
(petrosb@merl.com)} 
\address{(PTB)  Mitsubishi Electric Research Laboratories (MERL), Cambridge, MA 02139 USA.}
\email{petrosb@merl.com}

\thanks{A preliminary version of this work has been published at {ICASSP 2026}. In contrast to the conference version, the current paper provides complete statement and proof of the theoretical results, and extends numerical experiments.}
\thanks{This work was completed while H.~Z. and R.~K. were interns at MERL}
\thanks{\textsuperscript{$\dagger$} Corresponding author.}

\date{}
\hypersetup{hidelinks}
\begin{document}
\begin{abstract}
   This paper introduces an interacting-particle optimization method tailored to possibly non-convex composite optimization problems, which arise widely in signal processing. The proposed method, \emph{ProxiCBO}, integrates consensus-based optimization (CBO) with proximal gradient techniques to handle challenging optimization landscapes and exploit the composite structure of the objective function. We establish global convergence guarantees for the continuous-time finite-particle dynamics and develop an alternating update scheme for efficient practical implementation. Simulation results for signal processing tasks, including signal recovery from one-bit quantized measurements and parameter estimation from single-photon lidar data, demonstrate that ProxiCBO outperforms existing proximal gradient methods and CBO methods in terms of both accuracy and particle-efficiency.
\end{abstract}
\maketitle
\section{Introduction}
In this paper, we propose an interacting-particle method for solving \textit{composite optimization} problems of the form
\begin{gather}\label{obj}
    \min_{v\in\mathbb{R}^d} \left\{\mathcal{E}(v):=f(v)+g(v)\right\},
\end{gather}
where $f(v)$ is differentiable but possibly non-convex, and $g(v)$ is convex but possibly non-differentiable. 

Composite optimization provides a unifying framework for a wide range of inverse problems in signal processing. In this context, $f(v)$ is the data fidelity term that incorporates observation model and promotes measurement consistency. For example, $f(v)=\tfrac12\|\mathcal{A}(v)-y\|_2^2$ is a commonly used data loss when given measurements $y$ and observation model $\mathcal{A}$. When $\mathcal{A}$ is nonlinear, $f$ is usually non-convex. 
In terms of Bayesian inference, the aforementioned quadratic loss is the negative log-likelihood function for additive Gaussian measurement noise. When the measurement system involves photon counting or other random point processes, \HZnew{the} Poisson distribution is usually used as the noise model. For example, in single-photon lidar, $f$ is the (non-convex) negative log-likelihood function of a time-inhomogeneous Poisson process~\cite{shin2015photon, kitichotkul2023role}. The second term $g(v)$ is the regularizer that encodes prior knowledge about the underlying signal to be reconstructed. For example, it can be the indicator function of a box constraint for bounded signals, $\ell_1$-norm for sparse signals, and total variation~\cite{chambolle2010introduction} for images.

For composite optimization problems, a standard approach employs gradient-based methods such as proximal gradient descent~\cite{parikh2014proximal} and its variants~\cite{FISTA, MAPG}. However, there are well-known downsides to proximal-gradient-type methods. In practice, they can be sensitive to initialization and may become trapped in poor local minima. Consequently, obtaining high-quality solutions often requires using prior information to design good warm-starts. Theoretically, global convergence guarantees are largely limited to convex problems~\cite{polyak}. For non-convex objectives, one typically obtains only local convergence results or convergence to critical points~\cite{MAPG, proxi_conv_1}.

Consensus-based optimization (CBO), an interacting-particle method, has recently emerged as a promising approach for tackling non-convex optimization problems while admitting rigorous global convergence guarantees. Initially proposed in~\cite{cbo1,cbo2}, the CBO framework has since been extended and analyzed along several directions, leading to both algorithmic improvements and deeper theoretical insights.

In this paper, we build on the practical effectiveness and theoretical foundations of CBO to develop a variant specifically tailored to \HZnew{the} composite optimization \HZnew{problem~\eqref{obj}}, for which we also establish rigorous global convergence guarantees.

\subsection{Contributions}
The main contributions of this work are as follows:
\begin{itemize}
\item We propose \textit{ProxiCBO}, a consensus-based optimization method specifically tailored to composite optimization by integrating gradient information of the differentiable term and \HZnew{the} proximal operator of the convex term.
\item \HZnew{We provide theoretical analysis of ProxiCBO by extending and refining the analytical frameworks\YMnew{~\cite{converge_globally, gerber2023mean} originally developed for the standard CBO.} 
In particular, we establish well-posedness \YMnew{of the stochastic differential equations associated with ProxiCBO}, a quantitative finite-particle mean-field approximation, and global convergence of the mean-field dynamics under explicit conditions on the objective landscape and initialization. In deriving the mean-field approximation, we revisit several estimates in~\cite{gerber2023mean} and develop direct quantitative arguments that allow us to explicitly track the dependence of the constants on the initial distribution and other problem and algorithmic parameters. We further characterize how the proximal-gradient force enters the convergence rate \YMnew{of the mean-field dynamics}, including its beneficial effect \YMnew{when $f$ is strongly convex.} 
}
\item We numerically demonstrate the superior performance of ProxiCBO in signal processing examples, benchmarking against proximal gradient methods~\cite{parikh2014proximal, FISTA} and existing CBO methods~\cite{high_dim, BAE2022126726}. Specifically, we show that compared with existing CBO methods, adding structural information of the objective into the particle dynamics can lead to better particle-efficiency, and compared with running proximal gradient methods with multiple initializations independently, having particles exchange information at each iteration can lead to better performance with the same set of \YMnew{initial particles}.
\end{itemize}

\subsection{Related Work}

On the algorithmic side, the CBO framework has been extended to address different classes of problems. \cite{high_dim} introduces anisotropic noise to improve performance in high-dimensional settings, a modification that has since been widely adopted. \cite{memory} utilizes memory effect as well as gradient information to boost the performance when the objective is smooth. Hence, it is applicable to the special case of \eqref{obj} where $g$ is differentiable. Extensions to constrained optimization have been developed in \cite{reg_CBO,BAE2022126726,carrillo2024interacting,project_CBO, trillos2024cb}, demonstrating strong performance on challenging tasks such as phase retrieval. Mirror-map variants~\cite{bungert2025mirrorcbo} further broaden the scope of CBO, showing success in applications like sparse neural network training and constrained optimization. Since constrained optimization can be formulated as a composite problem with an indicator function of the constraint set, these CBO variants~\cite{reg_CBO,BAE2022126726,carrillo2024interacting,project_CBO,bungert2025mirrorcbo} are applicable to a special case of \eqref{obj} where $g$ is an indicator function, though they do not require $f$ to be differentiable. However, existing CBO variants are not tailored to the specific structure of the objective function defined in~\eqref{obj}.

On the theoretical side, many studies have investigated convergence guarantees for CBO. Two main topics have emerged. The first concerns the analysis of the single-particle mean-field system as an approximation to the interacting particle system. Such an approximation can be shown to be exact (in an appropriate sense) as the number of particles goes to infinity.
Early works \cite{cbo1, cbo2, high_dim} established convergence of the mean-field system but relied on restrictive assumptions, such as requiring the initial particle distribution to be highly concentrated around the global minimizer. 
A new proof technique was introduced in \cite{converge_globally}, where global convergence was obtained under mild local coercivity and initialization assumptions.
\HZnew{The second topic focuses on closing the gap between the finite-particle
system and its mean-field limit. An early result~\cite{fornasier2020consensus} was
established for CBO dynamics constrained to a compact hypersurface, where compactness plays an important
role in the analysis. In the unconstrained Euclidean setting,
\cite[Proposition~3.11]{converge_globally} established a quantitative
mean-field approximation using a coupling argument in the spirit of
Sznitman, with the estimate holding on an event of arbitrarily high
probability. This result was later strengthened in
\cite{gerber2023mean}, which removed such conditioning and established
pathwise quantitative propagation of chaos. More recently,
\cite{gerber2025uniform} established uniform-in-time convergence with
explicit dependence of the constants, assuming the objective is global Lipschitz, which is more restrictive than the local Lipschitz
assumption in~\cite{gerber2023mean}.}

\subsection{Organization}

The paper is organized as follows. Section~\ref{sec:methodology} introduces the proposed method, 
Section~\ref{Theory} presents the main theoretical global convergence guarantee, 
and Section~\ref{sec:numerical} reports numerical experiments on signal processing problems.

\subsection{Notations}\label{sec:notation}

For nonnegative variables $A$ and $B$, $A\lesssim B$ means there exists a constant $C>0$ such that $A\leq C\cdot B$. For vectors, $\|\cdot\|_p$ and $\|\cdot\|_\infty$ denote the \HZnew{$\ell_p$- and $\ell_\infty$ norms}, respectively; for matrices, $\|\cdot\|_{F}$ denotes the Frobenius norm. $B^\infty_R(v)$ and $B_R(v)$ are $\ell_\infty$ and $\ell_2$ balls of radius $R$ centered at $v$. $\delta_v$ denotes the Dirac measure at $v$. The Wasserstein $p$-distance between probability measures $\mu$ and $\nu$ is $\mathcal{W}_p(\mu,\nu)=\inf_{\pi\in\Pi(\mu,\nu)}\int \|u-v\|_2^p\,d\pi$, where $\Pi(\mu,\nu)$ is the set of all couplings of $\mu$ and $\nu$. We write $\mathcal{P}(\R^d)$ for the set of probability measures on $\R^d$, $\mathcal{P}_p(\R^d)$ for those with finite $p$-th moments, i.e., $\int \|v\|_2^p\,d\mu<\infty$, and $\mathcal{P}_{p,R}(\R^d)$ for those with $p$-th moments bounded by $R^p$. Notably, for $\mu \in \mathcal{P}_{p}(\R^d)$, one has $\mathcal{W}_p(\mu,\nu)
=\left(\inf_{\pi\in\Pi(\mu,\nu)}
\int \|u-v\|_2^p\,d\pi\right)^{1/p}$. For a measure $\rho$, $L^p(\rho)$ denotes the space of $L^p$-integrable functions, with norm $\|f\|_{L^p(\rho)}
=\left(\int |f(v)|^p\,d\rho(v)\right)^{1/p}$. For two topological spaces $X$ and $Y$, let $\mathcal{C}(X,Y)$ denote the space of continuous mappings from $X$ to $Y$. Throughout the paper, all Brownian motions are defined on a common filtered probability space $(\Omega, \{\mathcal{F}_t\}_{t\ge 0}, \mathbb{P})$.
 Finally, $\mathcal{A}(s,l)$ denotes the class of objectives satisfying conditions (ii) and (iv) in Assumption~\ref{assump} {stated in Section~\ref{Theory}}.

\section{Methodology}\label{sec:methodology}
Like many other CBO algorithms, our algorithm is a time discretization of a set of stochastic differential equations (SDE) that characterizes the dynamics of an interacting particle system.
We first present the continuous-time dynamics of the particle system and then discuss its discrete implementation.
\begin{algorithm}
\caption{\textit{ProxiCBO}}\label{algo:alter}
\begin{algorithmic}[1]
\State \textbf{Input:} \HZnew{$V^i\overset{\text{i.i.d.}}{\sim}\rho_0$}, $i=1,\ldots,N$, $D(\cdot)=\text{diag}(\cdot)$
    \While{not \textsc{Stop}}
        \State \textbf{Compute consensus point:}
                       \[
            v_\alpha \gets \frac{\sum_{i=1}^{N} \exp\!\left(-\alpha\,\mathcal{E}(V^i)\right) V^i}{\sum_{i=1}^N \exp\!\left(-\alpha\,\mathcal{E}(V^i)\right)}.
        \]
        \State \textbf{Update particles:} for $i=1,\ldots,N$
        \begin{align*}
        &V^i \gets \, V^i - \lambda_1 \left(V^i - v_\alpha\right)\Delta t + \sigma_1 D\!\left( V^i - v_\alpha \right) z^{i,1} \sqrt{\Delta t} \\
            & + \sigma_2 D\!\left( \frac{1}{\mu} \left[ V^i - \mathrm{prox}_{\mu g} 
                \left( V^i - \mu \nabla f(V^i) \right) \right] \right) z^{i,2} \sqrt{\Delta t},
        \end{align*}
        where $\{z^{i,1}\}_{i=1}^N$ and $\{z^{i,2}\}_{i=1}^N$ 
        are independent $d$-dimensional standard Gaussian vectors.
        \State \textbf{Apply proximal map:} for i=1,\ldots,N
        \[
            V^i \gets \mathrm{prox}_{\mu g}\!\left( V^i - \mu \nabla f(V^i) \right)
        \]
    \EndWhile
    \State \textbf{Output:} Historical or current best particle location
\end{algorithmic}
\end{algorithm}
\subsection{Continuous-time Dynamics}\label{sec:continuous}
We begin by describing the interacting-particle system. Let $V^{i,N}_t$ denote the position of the $i$-th particle in the $N$-particle ensemble at time $t$. At time $t=0$, $N$ particles $\{V^{i,N}_0\}_{i=1}^N$ are sampled independently from a common initial distribution $\rho_0$. The objective of the proposed dynamics is to drive the empirical measure
\begin{gather}\label{empirical measure}
\widehat{\rho}^N_t = \frac{1}{N}\sum_{i=1}^N \delta_{V^{i,N}_t}
\end{gather}
towards the Dirac measure at a global minimizer $v^*$ of~\eqref{obj}. The dynamics of the particle system are given by the SDE,
\begin{equation}\label{dyn}
\begin{aligned}
   &d V^{i,N}_t = \underbrace{-\lambda_1 \left(V^{i,N}_t - v_\alpha (\widehat{\rho}^N_t)\right)\,dt}_{T_1} \\
    &\underbrace{-\lambda_2 \left( \nabla f(V^{i,N}_t) + \nabla M_{\mu g} \left(V_t^{i,N} - \mu \nabla f(V^{i,N}_t)\right)\right)\,dt}_{T_2}\\
    & \underbrace{+ \sigma_1 D\left(V^{i,N}_t - v_\alpha (\widehat{\rho}^N_t)\right) \, dB_t^{i,1}}_{T_3}\\
    & \underbrace{+ \sigma_2 D\left( \nabla f(V^{i,N}_t) + \nabla M_{\mu g} \left(V_t^{i,N} - \mu \nabla f(V^{i,N}_t)\right)\right)\, dB^{i,2}_t}_{T_4}.
\end{aligned}
\end{equation}
Here $v_\alpha (\widehat{\rho}^N_t)$ is the \textit{consensus point} with respect to the empirical measure $\widehat{\rho}^N_t$. Given an objective function $\CE$, a probability measure $\rho$, and an inverse temperature $\alpha>0$, the consensus point with respect to $\rho$ is defined as
\begin{gather}\label{eq:def_consensus}
v_\alpha (\rho) = \frac{\int v \cdot \omega_\alpha(v) \,d\rho(v)}{\|\omega_\alpha\|_{L^1 (\rho)}},
\end{gather}
with the weight $\omega_\alpha  (v) = \text{exp}\left({-\alpha \mathcal{E}(v)}\right)$. In particular,  $$v_\alpha (\widehat{\rho}^N_t) = \frac{\sum_{i=1}^N  \omega_\alpha (V^{i,N}_t) V_t^{i,N}}{\sum_{i=1}^N \omega_\alpha (V_t^{i,N})}.$$
$M_{\mu g}$ is the Moreau envelope of $g$ with parameter $\mu >0$, defined as $$M_{\mu g}(v):= \inf_{u\in\mathbb{R}^d} \left\{g(u)+\frac{1}{2\mu} \|v - u\|_2^2\right\}.$$
$B^{i,1}_t$, $B_t^{i,2}$ are two independent $d$-dimensional Wiener processes, and $D:\mathbb{R}^d \rightarrow \mathbb{R}^{d\times d}$ is a map we shall introduce later. In the following, we explain each term in \eqref{dyn}. 
\paragraph{Term $T_1$} This drift term is inherited from vanilla CBO methods \cite{cbo1, cbo2}. It exploits the current information owned by particles, guiding all particles toward the \textit{consensus point} $v_\alpha (\widehat{\rho}^N_t)$. Motivated by the well-known Laplace principle \cite{M}, the consensus point $v_\alpha(\widehat{\rho}^N_t)$ smoothly approximates the particle with the lowest objective function value at the current iteration. \HZnew{The inverse temperature $\alpha$ controls the concentration of this weighted average: a larger $\alpha$ assigns relatively more weight to particles with smaller objective values, so that the consensus point becomes
more concentrated around the currently best-performing particles.} Consequently, the particles are encouraged by this term to gather around a location where the objective value is small, based on current information, where $\lambda_1 > 0$ controls the magnitude of this move.

\paragraph{Term $T_2$} Inspired by \cite{memory}, this drift term exploits first-order information of the objective function, in the spirit of proximal gradient descent. It is based on the observation in \cite{proximal_flow} that, under proper assumptions, \eqref{obj} can be solved using the proximal gradient flow,
\begin{equation}\label{proxi_flow}
\begin{aligned}
\frac{d v(t)}{d t} &=-\mu \left(\nabla f(v) + \nabla M_{\mu g} \left(v - \mu\nabla f(v)\right)\right).
\end{aligned}
\end{equation}
One can notice that the standard proximal gradient descent can be obtained via an explicit forward Euler discretization of \eqref{proxi_flow} with step size one. For each particle, a drift in the direction of $\nabla f(v) + \nabla M_{\mu g} \left(v - \mu\nabla f(v)\right)$ provides first-order information of the objective landscape, thereby augmenting the dynamics. Here, $\lambda_2>0$ controls the magnitude of this force.

\paragraph{Terms $T_3$ and $T_4$} The two diffusion terms facilitate exploration. The above two drift terms $T_1$ and $T_2$ are based on the current information obtained by particles. However, this risks incorporating biased information. For example, if the initialization of the particles is concentrated around a local minimizer of the objective function, particles will fail to explore the remaining landscape and concentrate at the local minimizer. 
These two diffusion terms help prevent this undesired situation. 
The matrix-valued function $D$ determines the way of exploration. The isotropic exploration \cite{cbo1} can be done by choosing $D(v) = \|v\|_2 I_d$, where $I_d$ is the $d\times d$ identity matrix, while choosing $D(v)=\text{diag}(v)$ gives the anisotropic exploration \cite{high_dim}. Throughout the paper, we adopt anisotropic diffusion in both the theoretical analysis and the numerical experiments. $\sigma_1>0$ and $\sigma_2>0$ are parameters that determine of willingness of exploration.

Intuitively, if $v^*$ is a global minimizer of~\eqref{obj}, then $\delta_{v^*}^{\otimes N}$ is a stationary distribution of the particle ensemble~\eqref{dyn}. Indeed, in this case $v_\alpha(\widehat{\rho}^N_t)=v^*$, so the consensus-related terms $T_1$ and $T_3$ vanish. Moreover, by properties of Moreau envelope, 
\begin{gather*}
\nabla f(v^*) + \nabla M_{\mu g}(v^* - \mu \nabla f(v^*)) = 0, \quad \forall \mu > 0,
\end{gather*}
which implies that the terms $T_2$ and $T_4$ also vanish. Consequently, the system stabilizes at $v^*$. In the subsequent analysis, we will show that for certain $f$ and $g$, and provided the initial distribution assigns positive mass in any neighborhood of $v^*$, the empirical measure of the particles concentrates around $v^*$ within a finite time.

\subsection{Practical Implementation}\label{sec:practical impelmentation}
In this section, we discretize the continuous-time dynamics~\eqref{dyn} and present the practical implementation.

A natural approach is to apply the Euler–Maruyama scheme to~\eqref{dyn}. \HZnew{For fixed $N$ and $\mu>0$, the standard Euler--Maruyama discretization converges strongly to the continuous-time particle system on any finite time interval under the assumptions considered in Section~\ref{Theory}. We provide the corresponding scheme and a convergence result in Supplementary~\ref{EM}, where the result follows by verifying the conditions of~\cite[Theorem~2.2]{higham2002strong}.} However, when $g$ includes an indicator function, the problem is constrained.
Under a naive Euler–Maruyama discretization, neither the drift nor the diffusion terms guarantee that the particles remain within the constraint set.

To address this issue, we adopt an alternating update scheme: a consensus step that discretizes terms $T_1$, $T_3$, and $T_4$, followed by a proximal step corresponding to $T_2$. This formulation is obtained by setting $\lambda_2 \Delta t = \mu$, so that the contribution of $T_2$ is enforced through a proximal update rather than appearing directly in the particle drift, thereby ensuring the particles to stay within \HZnew{the} constraint set.

The gradient of the Moreau envelope is computed using the proximal operator of $g$~\cite[Proposition 12.30]{bauschke2017}
\begin{gather}\label{eq:moreau_prox}
\nabla M_{\mu g}(v) = \tfrac{1}{\mu}\big(v - \mathrm{prox}_{\mu g}(v)\big),
\end{gather}
where $\mathrm{prox}_{\mu g}(v) = \argmin_{u \in \mathbb{R}^d} \Big\{ g(u) + \tfrac{1}{2\mu}\|v-u\|_2^2 \Big\}$. 

In practice, one may either record the historical best location (the lowest objective value encountered during the run) or simply use the best location at the final iteration as the output. The pseudocode of \textit{ProxiCBO} is summarized in Algorithm~\ref{algo:alter}.

\section{Theoretical Analysis}\label{Theory}

We now present the theoretical analysis of the proposed particle system~\eqref{dyn} with anisotropic diffusion (i.e., $D(v)=\text{diag}(v)$ in~\eqref{dyn}). 
We begin by introducing the constants that will be used in the analysis and stating our main convergence result, which characterizes the long-time behavior of the finite-particle system. 
The remainder of the section is devoted to establishing the ingredients needed for the proof.
\subsection{Constants}\label{sec:constants} 
Before presenting the main theoretical analysis, we introduce in this subsection the constants that will be used in the statements and proofs of the theorems in this section and the appendix. We define
\begin{gather*}
p_\CM := \begin{cases}s+2,&l=0\\ 1&l>0\end{cases},
\end{gather*} and $k_p(t):=\max\{t^p +t^{\frac{p}{2}},t^{p-1}+t^{\frac{p}{2} - 1}\}.$ $C$ denotes a generic constant depending only on the algorithm parameters ($\alpha,\lambda_1,\lambda_2,\sigma_1,\sigma_2,\mu$) that will be defined in~\eqref{dyn} and objective properties ($L_f,L_{\mathcal{E}},s,\underline{\mathcal{E}},l,c_l,C_l,c_u,C_u$) that will be defined in Assumption~\ref{assump}.
 Further, given $T$, $v^*$, and $V_0\sim \rho_0$, $C_{\text{mean}}(T)$ is defined as 
    \begin{align*}
C\left(T+T^{2}\right)\cdot e^{ C\cdot\left(T+T^{2}\right)\cdot \left(1+\Psi_{\rho_0,T,v^*}\right)} \cdot \Lambda_{\rho_0,T,v^*}
    \end{align*}
with
    \begin{align*}
 &\Psi_{\rho_0,T,v^*}:=\left(1 + e^{C(1+(2\sqrt{K+1})^{l} )}\left(1+(K+1)^{\tfrac{3}{2}}\right)\right)^2,\\
  &K:= C\cdot \left(\E \left[ \|V_0\|_2^2 \right]+k_2(T)\|v^*\|_2^2\right)\cdot e^{C\cdot T\cdot k_2(T)},\\
&\Lambda_{\rho_0,T,v^*} := \left(\E \left[ \|V_0\|_2^8 \right]+k_8(T)\|v^*\|_2^8\right)^{3/4}\cdot e^{C\cdot T\cdot k_8(T)}+  e^{\Phi_{\rho_0, T,v^*}}\left(\max\{k_4(T),k_2^2(T)\}\|v^*\|_2^4 + \E[\|V_0\|_2^4]\right)^{1/2},\\
&\Phi_{\rho_0,T,v^*} := C\Big(1+T\cdot \max\{k_2(T),k_4(T)\}+ \left(\E \left[ \|V_0\|_2^2 \right]+k_2(T)\|v^*\|_2^2\right)^{l/2}\cdot e^{C\cdot T\cdot k_2(T)}\Big).
\end{align*}
\subsection{Main Result}
We now present the main theorem of the paper. Our goal is to establish convergence of the particle system to the global minimizer $v^*$ of ~\eqref{obj}.
Specifically, we show that the empirical measure $\widehat{\rho}^N_t$ defined in \eqref{empirical measure} \HZnew{approaches} the Dirac measure at $v^*$ \HZnew{in the expected Wasserstein-2 sense}. Before stating the main theorem, we introduce the assumptions for the objective function $\CE(v) = f(v) + g(v)$ required for the analysis.
\begin{assumption}\label{assump}
(i) The objective function $\mathcal{E}$, with $f$ differentiable and $g$ convex, is bounded from below and admits a unique global minimizer $v^*\in\R^d$, with \HZnew{$
\underline{\mathcal{E}}:=\mathcal{E}(v^*)=\min_{v\in\R^d}\mathcal{E}(v).$}\\
(ii) There exist $s,L_\CE>0$ such that for all $u,v\in\R^d$, $|\CE(u) - \CE(v)|\leq L_\CE (1+\|u\|_2+\|v\|_2)^s\|u-v\|_2$.\\
(iii) $\nabla f$ is Lipschitz with Lipschitz constant $L_f>0$.\\
(iv) There exist $l\geq 0$ and $c_l,C_l,c_u,C_u>0$ such that for all $v\in\R^d$, $\CE(v)-\underline{\CE} \leq c_u \|v\|_2^l + C_u$ and $\CE(v)-\underline{\CE}\geq c_l \|v\|_2^l - C_l$.\\
\HZnew{(v) There exist $R_0,\CE_\infty,\eta,\nu>0$ such that
\begin{gather*}
\CE(v)-\CE(v^*)>\CE_\infty,
\qquad
\text{for all } v\in \left(B^\infty_{R_0}(v^*)\right)^c,
\end{gather*}
and
\begin{gather*}
\|v-v^*\|_\infty
\leq
\frac{1}{\eta}
\left(\CE(v)-\CE(v^*)\right)^\nu,
\qquad
\text{for all } v\in B^\infty_{R_0}(v^*).
\end{gather*}}
\end{assumption}
 
\HZnew{The uniqueness requirement in (i) is mainly imposed to enable the theoretical convergence analysis: when several global minimizers attain the same objective value, the exponential weighting does not in general distinguish among them, so convergence to a prescribed minimizer cannot be expected without additional assumptions. Assumptions (ii)--(iv) are mainly regularity and growth conditions that enable the well-posedness and mean-field approximation analysis in Sections~\ref{sec:well-posed} and~\ref{sec:mean-field}. Assumption~\ref{assump}(v), on the other hand, provides the main condition on the objective landscape for establishing global convergence. Similar conditions are commonly used in the global convergence analysis of CBO methods; see, e.g.,~\cite{converge_globally,memory}. In particular, the two conditions in Assumption~\ref{assump}(v) play complementary roles. The first condition guarantees that points sufficiently far away from $v^*$ have objective values separated from the global minimum by a fixed positive gap, while the second condition ensures that, around $v^*$, a sufficiently small objective gap implies that the corresponding point is close to $v^*$. The classical Laplace principle characterizes the minimal objective value through exponentially weighted integrals as $\alpha$ becomes large, but does not by itself provide a quantitative estimate on the spatial location of the weighted consensus point, namely, $v_\alpha(\rho)$. Assumption~\ref{assump}(v) bridges this gap by connecting the objective gap to the distance from $v^*$. The widely-used quantitative Laplace principle, e.g.,~\cite[Proposition~4.5]{converge_globally} or~\cite[Proposition~3.6]{memory}, combines these properties with the exponential weighting to control the distance between $v_\alpha(\rho_t)$ and $v^*$, which is the key step in the long-time convergence argument in Section~\ref{sec:long-time}.}

\HZnew{We emphasize that such additional conditions are needed when establishing a global convergence guarantee for general non-convex objectives. The assumptions above still allow the objective to be non-convex and to contain non-global local minima. However, without additional information on the objective landscape and the initialization, the dynamics cannot in general be expected to identify the global minimizer. For example, if the initial distribution is concentrated at a non-global stationary point $\bar v$ satisfying $
\bar v=\mathrm{prox}_{\mu g}\left(\bar v-\mu\nabla f(\bar v)\right),$ then $v_\alpha(\delta_{\bar v})=\bar v$ and $
\nabla f(\bar v)+\nabla M_{\mu g}\left(\bar v-\mu\nabla f(\bar v)\right)=0.$ Consequently, all drift and diffusion terms in~\eqref{SDE} vanish if $\rho_0=\delta_{\bar v}$, and the dynamics remain at $\bar v$. The condition $v^*\in\operatorname{supp}(\rho_0)$ in Theorem~\ref{main} avoids this situation by requiring the initial distribution to assign positive mass to every neighborhood of $v^*$. Notice that this does not require $\rho_0$ to be concentrated around the global minimizer.}

Under these assumptions, our main theoretical result follows:

\begin{theorem}\label{main}
    Let Assumption~\ref{assump} hold with $0<l\leq s + 1$ or $l=s=0$. \HZnew{Assume further that $\rho_0$ has bounded moments of all orders.} Choose parameters of the algorithm such that $2\lambda_1 - \sigma_1^2 - \lambda_2(2 L_f + \tfrac{1}{\mu}) - \sigma_2^2 (2 L_f + \tfrac{1}{\mu})^2>0$. Given any error tolerance $\delta>0$, assume \HZnew{$v^*\in\operatorname{supp}(\rho_0)$} and $\CW_2^2(\rho_0,\delta_{v^*}) >\delta$, there exists a choice of $\alpha>0$ {for the dynamics~\eqref{dyn}} such that
    \begin{equation}\label{mainbound}
    \begin{aligned}
    &\HZnew{
    \E\left[
    \min_{t\in[0,T^*]}
    \CW^2_2(\widehat{\rho}_t^N , \delta_{v^*})
    \right]
    }
    \leq \HZnew{2} C_{\text{mean}}\cdot N^{-1}+\frac{\delta}{2}.
    \end{aligned}
    \end{equation}
    where $T^*$ is defined as
    \begin{gather}\label{def of T star}
    \frac{2\log\left({4\CW^2_2(\rho_0 , \delta_{v^*})}/{\delta}\right)}{ \left(2\lambda_1 - \sigma_1^2 - \lambda_2(2 L_f + \tfrac{1}{\mu}) - \sigma_2^2 (2 L_f + \tfrac{1}{\mu})^2\right)}.
    \end{gather}
   \HZnew{Here, $\E$ is taken with respect to the randomness of the initial particles and the driving Brownian motions in~\eqref{dyn}.}  In particular, if the number of particles $N$ is greater than
    \begin{align*}
C\left(T^*+(T^*)^{2}\right)\cdot e^{C\cdot\left(T^*+(T^*)^{2}\right)\cdot \left(1+\Psi_{\rho_0,T^*,v^*}\right)} \cdot \Lambda_{\rho_0,T^*,v^*}/\delta,
    \end{align*}
    then \[
    \HZnew{
    \E\left[
    \min_{t\in[0,T^*]}
    \CW^2_2(\widehat{\rho}_t^N , \delta_{v^*})
    \right]
    }
    < \delta,\]
where $C$ represents a generic constant depending only on $\alpha,\lambda_1,\lambda_2,\sigma_1,\sigma_2,\mu,L_f,L_{\CE},s,\underline{\CE},l,c_l,C_l$. Here, $C_{\text{mean}}$, $\Psi_{\rho_0,T^*,v^*}$, $\Lambda_{\rho_0,T^*,v^*}$ are constants defined in Section~\ref{sec:constants}.
\end{theorem}
The bound \eqref{mainbound} does not depend explicitly on the dimension $d$. 
However, \HZnew{for a fixed choice of $\alpha$,} the first error term, which arises from the mean-field approximation, depends on $d$ through the moments of the initial distribution $\rho_0$ and the initial error $\int \|v - v^*\|_2^2\,d\rho_0$, as one can see in the definition of $C_{\text{mean}}$ in Section~\ref{sec:constants}. In typical settings, these quantities scale polynomially in $d$, resulting in \HZnew{a} dependence on $d$ that is doubly exponential, with the inner exponent being polynomial in $d$. This dependence is due to Proposition~\ref{prop:stab}, which gives a factor that grows exponentially with these moments. Then the application of Gr\"onwall's inequality produces a doubly exponential dependence. Moreover, the exponential dependence from Proposition~\ref{prop:stab} is not an artifact of the analysis: Remark~\ref{expoenntial_example} gives a construction where this rate is attained. 
\HZnew{We further note that $C_{\text{mean}}$ also depends on the inverse-temperature parameter $\alpha$, and the sufficiently large choice of $\alpha$ required for global convergence may itself depend on $d$. This dependence is made explicit in~\eqref{eq:alpha_delta} in the proof of Theorem~\ref{thm:mean-field long-time-behavior} (see also~\cite[Eq.~(3.15)]{converge_globally}).}

\HZnew{The convergence condition in Theorem~\ref{main} treats the
proximal-gradient force conservatively in order to accommodate non-convex
$f$. Under an additional strong-convexity assumption, this force instead
improves the convergence rate, as shown below.}
\HZnew{
\begin{proposition}\label{prop:strongly-convex-main}
Let Assumption~\ref{assump} hold with $0<l\leq s+1$ or $l=s=0$, and suppose
additionally that $f$ is $m_f$-strongly convex for some $m_f>0$. Assume further that $\rho_0$ has bounded moments of all orders.
Choose parameters of the algorithm such that $
0<\mu\leq\frac{2}{m_f+L_f},
\sigma_1,\sigma_2>0,
\sigma_2^2\leq\lambda_2\mu, $
and $
2\lambda_1-\sigma_1^2
+\frac{2\lambda_2m_fL_f}{m_f+L_f}>0.$ Given any error tolerance $\delta>0$, assume
$v^*\in\operatorname{supp}(\rho_0)$ and
$\CW_2^2(\rho_0,\delta_{v^*})>\delta$, there exists a choice of $\alpha>0$
for the dynamics~\eqref{dyn} such that
\begin{equation}\label{mainbound_sc}
\E\left[
\min_{t\in[0,T_{\mathrm{sc}}^*]}
\CW_2^2(\widehat{\rho}_t^N,\delta_{v^*})
\right]
\leq
2C_{\mathrm{mean}}\cdot N^{-1}
+\frac{\delta}{2},
\end{equation}
where $T_{\mathrm{sc}}^*$ is defined as
\begin{equation}\label{def of T star sc}
T_{\mathrm{sc}}^*
=
\frac{
2\log\left(
4\CW_2^2(\rho_0,\delta_{v^*})/\delta
\right)}
{
2\lambda_1-\sigma_1^2
+\frac{2\lambda_2m_fL_f}{m_f+L_f}
}.
\end{equation}
Here, $\E$ is taken with respect to the randomness of the initial particles
and the driving Brownian motions in~\eqref{dyn}.
In particular, if the number of particles $N$ is greater than
\[
C\left(T_{\mathrm{sc}}^*+(T_{\mathrm{sc}}^*)^2\right)
\cdot
e^{C\cdot
\left(T_{\mathrm{sc}}^*+(T_{\mathrm{sc}}^*)^2\right)
\cdot
\left(1+\Psi_{\rho_0,T_{\mathrm{sc}}^*,v^*}\right)}
\cdot
\Lambda_{\rho_0,T_{\mathrm{sc}}^*,v^*}/\delta,
\]
then \[
\E\left[
\min_{t\in[0,T_{\mathrm{sc}}^*]}
\CW_2^2(\widehat{\rho}_t^N,\delta_{v^*})
\right]
<\delta,\]
where $C$ represents a generic constant depending only on
$\alpha,\lambda_1,\lambda_2,\sigma_1,\sigma_2,\mu,L_f,L_{\CE},
s,\underline{\CE},l,c_l,C_l$.
Here, $C_{\mathrm{mean}}$,
$\Psi_{\rho_0,T_{\mathrm{sc}}^*,v^*}$,
$\Lambda_{\rho_0,T_{\mathrm{sc}}^*,v^*}$ are constants defined in
Section~\ref{sec:constants}.
\end{proposition}
}
\HZnew{\begin{remark}\label{remark:strongly-convex-main}
Proposition~\ref{prop:strongly-convex-main} shows that the adverse
dependence on $\lambda_2$ in Theorem~\ref{main} is a consequence of covering possibly non-convex $f$. Under strong convexity, the proximal-gradient force
instead contributes the strictly positive quantity $
\frac{2\lambda_2m_fL_f}{m_f+L_f}$ to the decay rate reflected in~\eqref{def of T star sc}. Therefore, 
compared with the general result in Theorem~\ref{main}, the \HZnew{mean-field} exponential decay is faster and the time to reach the same prescribed accuracy is reduced from $T^*$ to
$T_{\mathrm{sc}}^*<T^*$. This demonstrates the benefit of
exploiting the proximal-gradient information when it reliably points to
the global minimizer.
\end{remark}}

The proof of Theorem~\ref{main} proceeds in two steps. First, we
approximate the finite-particle system by its mean-field limit:
\HZnew{we construct a synchronous coupling between the interacting
particles and corresponding copies of the mean-field process and
derive a quantitative particlewise approximation estimate}; see
Section~\ref{sec:mean-field}. Second, we analyze the long-time behavior
of the mean-field dynamics and show convergence to the global minimizer
by studying $\CW_2(\rho_t,\delta_{v^*})$; see
Section~\ref{sec:long-time}. Combining these two ingredients yields the
\HZnew{expected} finite-particle global convergence stated in
Theorem~\ref{main}.

In the remainder of this section, we present the components required to establish Theorem~\ref{main}, and present the proof for Theorem~\ref{main}.
Section~\ref{sec:well-posed} establishes the well-posedness of the proposed dynamics. Section~\ref{sec:mean-field} introduces the mean-field dynamics, proves
their well-posedness, and \HZnew{establishes a quantitative
particlewise approximation estimate through synchronous coupling}.
Finally, Section~\ref{sec:long-time} analyzes the long-time behavior of $\rho_t$ by studying $\mathcal{W}_2(\rho_t,\delta_{v^*})$. Combining the results obtained in Sections~\ref{sec:well-posed}, \ref{sec:mean-field}, \ref{sec:long-time}, we derive Theorem~\ref{main} in Section~\ref{proof of main}.

\HZnew{We postpone the proof of Proposition~\ref{prop:strongly-convex-main} to Appendix~\ref{appen: sc} for details as it largely follows the proof of Theorem~\ref{main}.}

\subsection{Well-posedness of~\eqref{dyn}}\label{sec:well-posed}
The first step is to establish that the proposed dynamics~\eqref{dyn} are well-posed. 
The following result provides this guarantee.
\begin{theorem}\label{thm:wellposed for micro}
Let Assumption~\ref{assump} (ii) \HZnew{and (iii)} hold. Then the SDE \eqref{dyn} has unique strong solutions for any initial condition that is independent of the Brownian \HZnew{m}otions. The solutions are almost surely continuous.
\end{theorem}
\begin{proof}
Please see Appendix~\ref{app:proof of micro}.
\end{proof}

\subsection{Mean-field Approximation}\label{sec:mean-field}
We now turn to the mean-field limit. Formally, as the number of particles $N \to \infty$, the particles become exchangeable and indistinguishable, and the evolution of the system can be described by the single mean-field SDE,
\begin{equation}\label{SDE}
\begin{aligned}
    d \overline{V}_t =& -\lambda_1 \left(\overline{V}_t - v_\alpha (\rho_t)\right)\,dt \\
    &-\lambda_2 \left( \nabla f(\overline{V}_t) + \nabla M_{\mu g} \left(\overline{V}_t - \mu \nabla f(\overline{V}_t)\right)\right)\,dt\\
    & + \sigma_1 D\left(\overline{V}_t - v_\alpha (\rho_t)\right) \, dB_t\\
    & + \sigma_2 D\left( \nabla f(\overline{V}_t) + \nabla M_{\mu g} \left(\overline{V}_t - \mu \nabla f(\overline{V}_t)\right)\right)\, d\widetilde{B}_t,
\end{aligned}
\end{equation}
where $\rho_t$ is the law of $\overline{V}_t$. Here, $\overline{V}_t$ characterizes the swarm behavior of the particles, and its law $\rho_t$ can be characterized by the following Fokker-Planck equation, 
\begin{equation}\label{FPK}
\begin{aligned}
	\partial_{t}\rho_{t}=&\nabla \cdot \left\{\left[\lambda_1 (v-v_\alpha(\rho_t))\right]\rho_t\right\}\\
    &+\nabla \cdot \left\{\left[\lambda_2 \left(\nabla f(v) + \nabla M_{\mu g}(v - \mu\nabla f(v))\right)\right]\rho_t\right\}\\
    &+ \frac{1}{2} \sigma_1^2 \sum_k \frac{\partial^2}{\partial x^2_k}\left((v - v_\alpha (\rho_t))^2_k \cdot \rho_t\right) \\
    &+\frac{1}{2} \sigma_2^2 \sum_k \frac{\partial^2}{\partial x^2_k}\left((\nabla f(v) + \nabla M_{\mu g}(v - \mu\nabla f(v)))^2_k \cdot \rho_t\right).
\end{aligned}
\end{equation}
\begin{remark}
\HZnew{Equation~\eqref{FPK} is understood in the weak sense, i.e., after pairing with a test function $\phi\in\mathcal{C}_c^2(\R^d)$ and transferring the spatial derivatives onto $\phi$ by integration by parts. Thus, the derivatives in~\eqref{FPK} are understood distributionally, and no classical differentiability of the coefficients or $\rho_t$ is required. As in~\cite[Remark~2.3]{memory}, one can extend the test-function space to non-compactly supported functions with suitable polynomial growth, such as $\phi=\frac{1}{2}\|v-v^*\|_2^2$ used in our proofs.}
\end{remark}
The next theorem ensures the mean-field SDE~\eqref{SDE} is well-posed.
\begin{theorem}\label{thm:mean-fild existence theorem}
Let Assumption~\ref{assump} (ii), \HZnew{(iii),} and (iv) hold with $l\leq s+1$ and fix a final time $T$. Assume $\rho_0 \in \CP_{p}(\R^d)$ for $p\geq \max\{2,p_{\CM} (s,l) \}$, and let $\overline{V}_0 \sim \rho_0$. Let $(\Omega, \{\mathcal{F}_t\}_{t\ge 0}, \mathbb{P})$ be the filtered probability space where Brownian motions are defined. Then there exists a strong solution $\overline{V}:\Omega \rightarrow \mathcal{C} ([0,T],\R^d)$ to \eqref{SDE} with initial condition $\overline{V}_0$ such that $t\rightarrow v_{\alpha} (\rho_t)$ is continuous over $[0,T]$, and it holds that 
\begin{equation}\label{mean-field-bound}
\begin{aligned}
\E \left[ \sup_{t\in [0,T]} \|\overline{V}_t\|_2^p\right]\leq  C\cdot \left(\E \left[ \|\overline{V}_0\|_2^p \right]+k_p(T)\|v^*\|_2^p\right)\cdot e^{C\cdot T\cdot k_p(T)},
\end{aligned}
\end{equation}
where $\rho_t = Law(\overline{V}_t)$, and $k_p(t)$ is defined in Section~\ref{sec:constants}. Further, the function $t\rightarrow \rho_t$ belongs to $\mathcal{C}([0,T],\CP_p (\R^d))$ and is a weak solution to \eqref{FPK}.
\end{theorem}
\begin{proof}
Please see Appendix~\ref{app:proof of macro}.
\end{proof}

Furthermore, the below theorem guarantees that $\rho_t$ as the law of $\overline{V}_t$ approximates $\widehat{\rho}^N_t$ well when $N$ is large enough.

\begin{theorem}\label{thm:quantitative-mean-field-limit}
Let Assumption~\ref{assump} (ii), \HZnew{(iii),} and (iv) hold with $0<l\leq s+1$ or $l=s=0$, and $V_0\sim\rho_0 \in \CP (\R^d)$ has bounded moments of all orders. Moreover, let $\{V_t^{i,N}\}_{i=1}^N$ be the solution to \eqref{dyn} with
\HZnew{$V_0^{i,N}\overset{\text{i.i.d.}}{\sim}\rho_0$, $i=1,\ldots,N$}, and let
\HZnew{$\{\overline{V}_t^{i,N}\}_{i=1}^N$ be copies of the solution
to \eqref{SDE} synchronously coupled with $\{V_t^{i,N}\}_{i=1}^N$,
in the way that $\overline{V}_0^{i,N}=V_0^{i,N}$ and
$\overline{V}^{i,N}$ is driven by the same Brownian motions.} Then
\begin{align*}
\E \left[\sup_{t\in[0,T]}\left\|V_{t}^{i,N} - \overline{V}_{t}^{i,N}\right\|_2^2\right]\leq C\left(T+T^{2}\right)\cdot e^{ C\cdot\left(T+T^{2}\right)\cdot \left(1+\Psi_{\rho_0,T,v^*}\right)} \cdot \Lambda_{\rho_0,T,v^*}\cdot N^{-1}.
\end{align*}
where $\Psi_{\rho_0,T,v^*}$ and $\Lambda_{\rho_0,T,v^*}$ are constants independent with $N$ and defined in Section~\ref{sec:constants}.

\end{theorem}
\begin{proof}
Please see Appendix~\ref{app:proof of quantitative}.
\end{proof}

\subsection{Long-time Behavior of the Mean-field System}\label{sec:long-time}
Since $\rho_t$ serves as a good approximation to $\widehat{\rho}^N_t$, it suffices to analyze the long-time behavior of $\rho_t$. This is described by the following theorem.
 \begin{theorem}\label{thm:mean-field long-time-behavior}
Let $\rho_t$ denote the law of the mean-field system \eqref{SDE} and assume Assumption~\ref{assump} (i), \HZnew{(ii)}, (iii) and (v) hold. Suppose the algorithmic parameters satisfy $\sigma_1,\sigma_2>0$ and $2\lambda_1 - \sigma_1^2 - \lambda_2(2 L_f + \tfrac{1}{\mu}) - \sigma_2^2 (2 L_f + \tfrac{1}{\mu})^2>0$. Fix any tolerance $\delta>0$. If \HZnew{$v^*\in\operatorname{supp}(\rho_0)$} and $\CW_2^2(\rho_0,\delta_{v^*}) > \delta$, then there exists $\alpha>0$ such that $$
            \min_{t\in [0,T^*]} \CW_2^2(\rho_t,\delta_{v^*}) \leq \delta,$$ where $T^*$ is defined in~\eqref{def of T star} 
        and $\CW_2^2(\rho_t,\delta_{v^*})$ has exponential decay before reaching $\delta$,
        \begin{align*}
            \CW_2^2(\rho_t,\delta_{v^*})\leq\CW_2^2(\rho_0,\delta_{v^*}) e^{{-\frac{1}{2} \left(2\lambda_1 - \sigma_1^2 - \lambda_2(2 L_f + \tfrac{1}{\mu}) - \sigma_2^2 (2 L_f + \tfrac{1}{\mu})^2\right)t}}.
        \end{align*}
    \end{theorem}
    \begin{proof}
   \HZnew{Please see Appendix~\ref{app:proof of long time}.}
    \end{proof}

    \subsection{Proof of the {Main Result}}
    \label{proof of main}
    Theorem~\ref{main} follows from Theorems~\ref{thm:quantitative-mean-field-limit}~and~\ref{thm:mean-field long-time-behavior} using the argument below.
    \begin{proof}[Proof of Theorem~\ref{main}]
\HZnew{
By Theorem~\ref{thm:mean-field long-time-behavior}, applied with tolerance
$\delta/4$, there exist $\alpha>0$ and a deterministic
$t_\delta\in[0,T^*]$ such that $
    \CW_2^2(\rho_{t_\delta},\delta_{v^*})\leq \delta/4,$
where $T^*$ is defined in~\eqref{def of T star}.}

\HZnew{Let $\{\overline V^{i,N}\}_{i=1}^N$ be the mean-field particles
synchronously coupled with $\{V^{i,N}\}_{i=1}^N$ as in
Theorem~\ref{thm:quantitative-mean-field-limit}. In particular,
$\operatorname{Law}(\overline V_t^{i,N})=\rho_t$ for every $i$ and $t$.
Since $\delta_{v^*}$ is a Dirac measure, we have
\begin{align*}
\E\left[
\CW_2^2(\widehat\rho_{t_\delta}^N,\delta_{v^*})
\right]
&=
\frac{1}{N}\sum_{i=1}^N
\E\left[\|V_{t_\delta}^{i,N}-v^*\|_2^2\right]\\
&\leq
\frac{2}{N}\sum_{i=1}^N
\E\left[
\|V_{t_\delta}^{i,N}-\overline V_{t_\delta}^{i,N}\|_2^2
\right] +
\frac{2}{N}\sum_{i=1}^N
\E\left[
\|\overline V_{t_\delta}^{i,N}-v^*\|_2^2
\right]\\
&\leq
2C_{\mathrm{mean}}N^{-1}
+
2\CW_2^2(\rho_{t_\delta},\delta_{v^*})\\
&\leq
2C_{\mathrm{mean}}N^{-1}+\frac{\delta}{2},
\end{align*}
where in the second inequality we used Theorem~\ref{thm:quantitative-mean-field-limit}.}

\HZnew{Finally, since $t_\delta\in[0,T^*]$, it holds pathwise that $
\min_{t\in[0,T^*]}
\CW_2^2(\widehat\rho_t^N,\delta_{v^*})
\leq
\CW_2^2(\widehat\rho_{t_\delta}^N,\delta_{v^*}).$ Taking expectations yields
\[
\E\left[
\min_{t\in[0,T^*]}
\CW_2^2(\widehat\rho_t^N,\delta_{v^*})
\right]
\leq
2C_{\mathrm{mean}}N^{-1}+\frac{\delta}{2},
\]
which proves~\eqref{mainbound}. The final claim follows directly by plugging in
the stated lower bound on $N$.
}
\end{proof}
    

\section{Numerical Experiments}\label{sec:numerical}

In this section, we compare the empirical performance of our algorithm with those of proximal gradient (PG)~\cite{parikh2014proximal}, accelerated proximal gradient (APG)~\cite{FISTA} and existing CBO-type algorithms~\cite{high_dim, BAE2022126726} for two signal reconstruction examples. All algorithms use the same initial particles and the final result for each algorithm is selected as the particle with the lowest objective function value. Note that PG and APG can be seen as particle systems without interactions. Hyper-parameters of the algorithms are tuned empirically and all algorithms use the same stepsize. Our first metric for quantifying performance is the success rate of \HZnew{attaining an objective value sufficiently close to the global minimum}, as this is important for evaluating an optimization algorithm. For each trial, we estimate the global minimum by running PG initialized at the ground truth signal and a trial is said to be successful if the excess error in objective function value above the estimated global minimum is smaller than a threshold. Specifically, let $v^*$ be the estimated global minimizer and let $\widehat{v}$ be the reconstructed signal, then a trial is successful if
\begin{equation}\label{eq:success_rate}
    \frac{\mathcal{E}(\widehat{v}) - \mathcal{E}(v^*)}{\mathcal{E}(v^*)} < 10^{-3}.
\end{equation}
Our second metric is related to mean squared error with respect to the ground truth signal, as this is important for signal processing applications. The specific definitions will be provided below in each example. Note that the ground truth is not necessarily the global minimizer due to measurement model mismatch, measurement noise, and estimation bias induced by the regularizer. 

\subsection{Example 1: One-Bit Signal Quantization}
Our first example is the non-monotonic quantization problem first proposed and analyzed in~\cite{boufounos2011universal}.
{Variants of this problem appear in efficient lightweight compression applications~\cite{Valsesia2016mar, Goukhshtein2020dec}.}
Consider the measurement model
\begin{equation}\label{sparse_recovery_meas}
    y=\text{sign}\left(\sin\left(\omega\left(Ax_0+u\right)\right)\right),
\end{equation}
where $x_0\in\mathbb{R}^d$ is the signal to be reconstructed, the measurement matrix $A\in\mathbb{R}^{m\times d}$ has i.i.d. Gaussian entries with mean zero and variance $1/d$, $\omega=\pi/\Delta$ with $\Delta$ the quantization bin-size, $\sin(\cdot)$ and $\text{sign}(\cdot)$ are applied element-wise to the arguments, and the vector $u\in\mathbb{R}^m$ is a known i.i.d. uniform dither taking values in $[-\Delta/2, \Delta/2]$. Consistent reconstruction is a combinatorial non-differentiable problem that may be infeasible in the presence of measurement errors or noise. We relax it, replacing $\text{sign}(\cdot)$ consistency with an $\ell_2$ data cost:
\begin{equation}\label{eq:one_bit_data_cost}
\mathcal{D}(x;A,y,u)=\frac{1}{2}\left\|y - \sin\left(\omega\left(Ax+u\right)\right)\right\|_2^2.
\end{equation}
Even with this relaxation, the problem has a very difficult optimization landscape, providing a good test case for optimization algorithms. To date, there are no good solutions known, unless a good estimate of the signal already exists (e.g., solving a hierarchy of problems~\cite{B_SampTA11}).
\paragraph{Sparse recovery}
When $x_0$ is known to be sparse, we use $\ell_1$ norm as the regularizer and estimate $x_0$ by solving
\begin{equation*}
    \min_{x\in\mathbb{R}^d}\, \left\{\mathcal{E}(x):= \mathcal{D}(x;A,y,u) + \lambda \|x\|_1\right\}.
\end{equation*}
In our simulations, $x_0$ has dimension $d=200$ with sparsity $s=10$, and the number of measurements is $m=4d$. All initial particles have i.i.d. standard normal entries. The results are computed from 500 trials, and in each trial, $A,x_0,u$ are independently sampled according to their distributions. 

\begin{figure}
    \centering
    \begin{subfigure}[b]{0.34\linewidth}
    \includegraphics[width=1.0\linewidth]{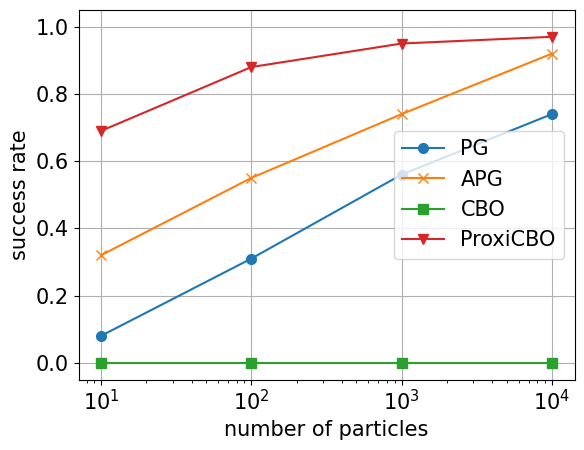}
    \caption{One-bit signal quantization.}
    \label{fig:sparse_success_rate}
    \end{subfigure}
    \begin{subfigure}[b]{0.34\linewidth}
    \includegraphics[width=1.0\linewidth]{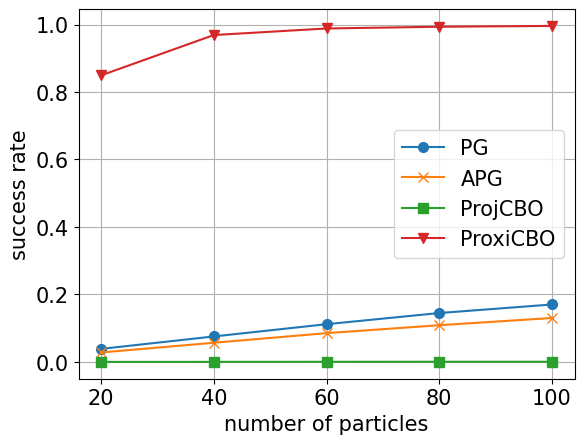}
    \caption{Single-photon lidar.}
    \label{fig:lidar_success_rate}
    \end{subfigure}
    \caption{Success rate by objective function value defined in \eqref{eq:success_rate}.}
    \label{success_rate}
\end{figure}

Fig.~\ref{fig:sparse_success_rate} compares the success rate \eqref{eq:success_rate} for PG, APG, CBO~\cite{high_dim} and ProxiCBO with $\lambda=0.25 \|y\|_2^2$ and $\omega=14$. We can see that ProxiCBO with 1000 particles outperforms other methods with 10,000 particles, showing ProxiCBO's superior particle-efficiency.
Fig.~\ref{fig:sparse_snr} compares the reconstruction signal to noise ratio (SNR), which is defined as $10 \log_{10}(\|x_0\|_2^2/\|\widehat{x}-x_0\|_2^2)$. 
In Fig.~\ref{fig:vary_lam}, we fix $\omega=14$ and change $\lambda$. A larger $\lambda$ can improve the optimization landscape, but may lead to a larger measurement mismatch, driving the optimizer away from the ground truth signal. Fig.~\ref{fig:vary_lam} shows that ProxiCBO performs the best for a wide range of $\lambda$ values. 
In Fig.~\ref{fig:vary_w}, we fix $\lambda=0.3\|y\|_2^2$ and change $\omega$. As $\omega$ increases (thus $\Delta$ decreases), the theoretical reconstruction error decreases~\cite{boufounos2011universal}, but the optimization landscape becomes more challenging. Fig.~\ref{fig:vary_w} shows that all methods perform well for small $\omega$ and ProxiCBO outperforms other methods as $\omega$ increases.

\begin{figure}
    \centering
    \begin{subfigure}[b]{0.34\linewidth}
    \includegraphics[width=1.0\linewidth]{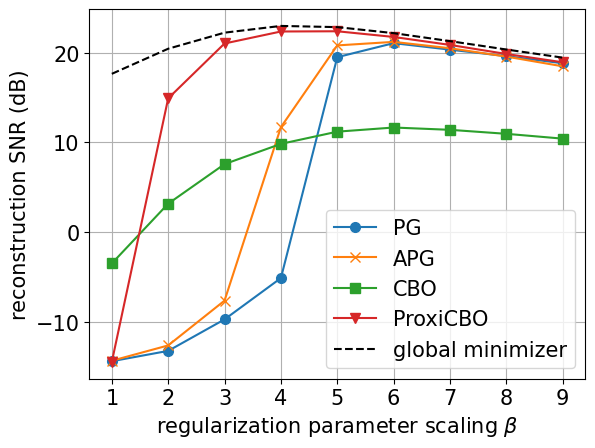}
    \caption{$\lambda=\beta\cdot 0.05\|y\|_2^2$ }
    \label{fig:vary_lam}
    \end{subfigure}
    \begin{subfigure}[b]{0.34\linewidth}
    \includegraphics[width=1.0\linewidth]{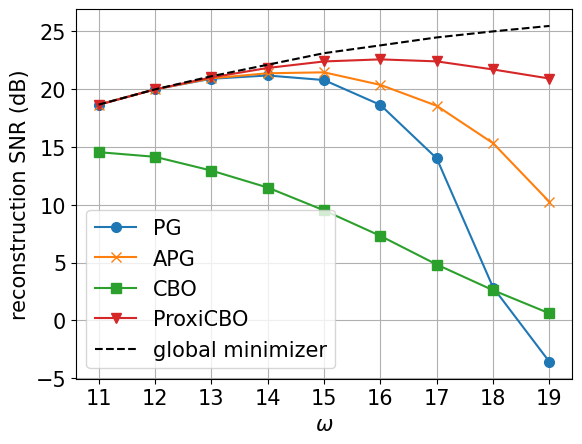}
    \caption{$\Delta=\pi/\omega$ }
    \label{fig:vary_w}
    \end{subfigure}
    \caption{Reconstruction SNR for sparse signal recovery with one-bit quantized measurements. The number of particles is 1000.}
    \label{fig:sparse_snr}
\end{figure}

\paragraph{Image recovery}
When $x_0$ is a vectorized image, we use constrained total variation (TV)~\cite{beck2009fast} as the regularizer and reconstruct the image by solving 
\begin{equation*}
    \min_{x\in\mathbb{R}^{d_x\times d_y}}\, \left\{\mathcal{E}(x):= \mathcal{D}(\text{vec}(x);A,y,u) + \lambda \text{TV}(x) + \iota_\mathcal{B}(x)\right\},
\end{equation*}
where $\text{vec}:\mathbb{R}^{d_x\times d_y}\to\mathbb{R}^d$ is the vectorization operator, $\text{TV}(\cdot)$ is the total variation semi-norm~\cite{chambolle2010introduction}, $\mathcal{B}$ is a box constraint for pixel values, and $\iota_\mathcal{B}$ is the indicator function for $\mathcal{B}$.

In our simulation, the ground truth image $I_0$ is a $64\times64$ Shepp-Logan phantom with pixel values in $[0,1]$, thus $x_0=\text{vec}(I_0)\in\mathbb{R}^{4096}$ and $\mathcal{B}=[0,1]^{64\times 64}$. We use $m=4d$ measurements and $\omega=12$. Fig.~\ref{fig:image_snr} compares the reconstruction SNR achieved by PG, APG, projected CBO (projCBO)~\cite{BAE2022126726} and proxiCBO with different regularization parameters. The results are averaged over 100 trials. In each trial, $A,u$ are independently sampled and $x_0$ is fixed among all trials.  For PG, APG, and proxiCBO, the proximal operator for constrained TV, $g(x)=\lambda \text{TV}(x) + \iota_\mathcal{B}(x)$, is computed using the method proposed in \cite{beck2009fast}. For projCBO, the objective function is $\mathcal{D}(\text{vec}(x);A,y,u) + \lambda \text{TV}(x)$ and the projector is defined for $\mathcal{B}$. We can see that $\beta=4$ results in the highest SNR achieved by the estimated global minimizer and that proxiCBO approaches the optimal value at $\beta=4$.

\subsection{Example 2: Single-Photon Lidar} 
In a typical single-photon lidar (SPL) setup, a target is illuminated by a pulsed laser, the reflected light is detected by a single-photon detector, and the photon detection times are recorded by a timing system. From those detection times, we can estimate the reflectivity $S$ and distance $z$ of the target. If the target is moving, we can also estimate its velocity $v$.
Suppose the pulse shape of the laser is defined by $h(t)$, which is normalized such that $\int_{-\infty}^{\infty} h(t)\,dt=1$. Let $\{t_k\}_{k=1}^K$ be the timestamps when the laser pulses are sent out, which are randomly generated in our simulations.

\paragraph{Static SPL} Assuming that the target is static, the photon detection process is a time-inhomogeneous Poisson process with intensity function
\begin{equation*}
    \lambda(t) = S  \sum_{k=1}^K h\left(t-\tau-t_k\right) + b,
\end{equation*}
where $b$ is the background intensity and $\tau=2z/c$ with $c$ being the speed of light is the time-of-flight (TOF). The log-likelihood function for estimating $(S,b,\tau)$ is defined as
\begin{equation*}
    \mathcal{L} = -SK - bt_a + \sum_{t\in\mathcal{T}} \log\left( S \sum_{k=1}^K h\left(t-\tau-t_k\right) + b\right),
\end{equation*}
where $t_a$ is the acquisition time and $\mathcal{T}$ is a set of detection times. Given $\mathcal{T}$, we can estimate $S$, $b$ and $\tau$ (thus $z$) by solving~\cite{shin2015photon, kitichotkul2023role}
\begin{equation}{\label{SPL_ML}}
    \min_{S,b, \tau}\, \left\{\mathcal{E}(S,b,\tau):= -\mathcal{L}(\mathcal{T};S,b,\tau) + \iota_{\mathcal{C}}(S,b,\tau)\right\},
\end{equation}
where $\mathcal{C}$ is the feasible set for $(S,b,\tau)$. In our simulations, $K=500$, $S_\text{true}=0.1$, $b_\text{true}=10^{-4}$, $\tau_\text{true}=234$ ns, $t_a=5\times10^5$ ns, thus the signal to background ratio (SBR) is $(K\cdot S)/(b\cdot t_a)=1$. The pulse shape $h(t)$ is the probability density function of the Gaussian distribution with mean zero and standard deviation $0.1$ ns. The feasible set $\mathcal{C}=[10^{-8},10]\times[10^{-8},10]\times[0,\infty)$. The initial particles are i.i.d. uniform in $[0,1]\times[0,1]\times[0,500]$.

\begin{figure}[tbp]
  \centering
  \includegraphics[width=0.34\linewidth]{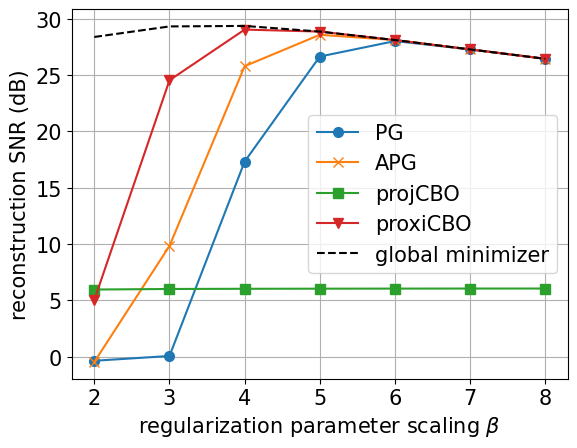}
  \caption{Reconstruction SNR for image recovery with
  one-bit quantized measurements. The number of particles
  is 500 and the regularization parameter is
  $\lambda=\beta\cdot 10^{-4}\|y\|_2^2$.}
  \label{fig:image_snr}
\end{figure}

Fig.~\ref{fig:lidar_success_rate} compares the success rate \eqref{eq:success_rate} for PG, APG, projected CBO~\cite{BAE2022126726}, and ProxiCBO. Fig.~\ref{fig:lidar_rmse} compares the root mean squared error (RMSE) for each of the parameters $S$, $b$, and $\tau$, where the RMSE for $S$ is defined as $\sqrt{\frac{1}{M}\sum_{i=1}^M (\widehat{S}_i-S_\text{true})^2}$ with $M$ being the number of trials and the definition of RMSE for other parameters is similar. We also include the Cram{\'e}r-Rao lower bound (CRB) in the plots showing the best achievable RMSE for any unbiased estimators~\cite{kay1993fundamentals}.
The results are computed from $M=10,000$ independent trials. The results show that ProxiCBO has better particle-efficiency than all comparison methods. Moreover, with sufficient particles, ProxiCBO can accurately solve the maximum likelihood problem~\eqref{SPL_ML} and achieve the CRB.
\begin{figure}
    \centering
    \begin{subfigure}[b]{0.34\linewidth}
    \includegraphics[width=1.0\linewidth]{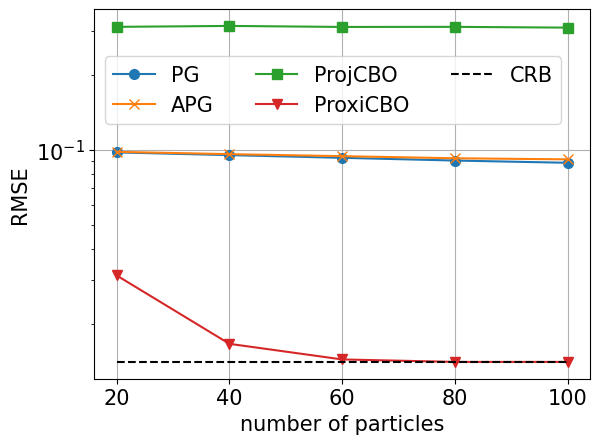}
    \caption{RMSE for $S$ estimation.}
    \label{fig:lidar_rmse_S}
    \end{subfigure}
    \begin{subfigure}[b]{0.34\linewidth}
    \includegraphics[width=1.0\linewidth]{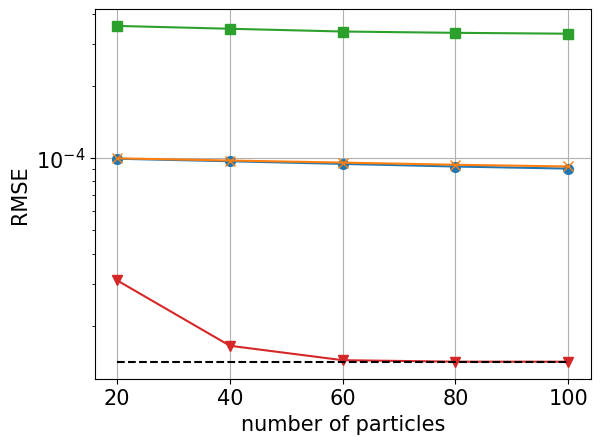}
    \caption{RMSE for $b$ estimation.}
    \label{fig:lidar_rmse_b}
    \end{subfigure}
    \begin{subfigure}[b]{0.34\linewidth}
    \includegraphics[width=1.0\linewidth]{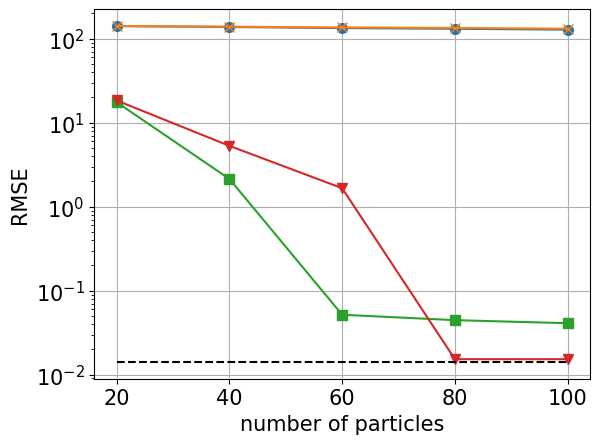}
    \caption{RMSE for $\tau$ estimation.}
    \label{fig:lidar_rmse_tau}
    \end{subfigure}
    \begin{subfigure}[b]{0.34\linewidth}
    \includegraphics[width=1.0\linewidth]{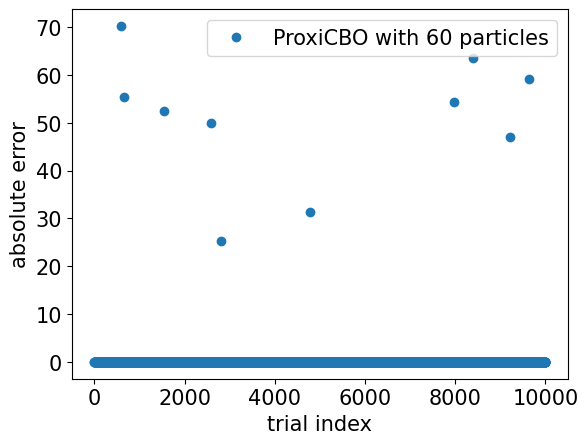}
    \caption{Error for $\tau$ in each trial.}
    \label{fig:lidar_tau_err}
    \end{subfigure}
    \caption{Estimation of $S$ (Fig.~\ref{fig:lidar_rmse_S}), $b$ (Fig.~\ref{fig:lidar_rmse_b}) and $\tau$ (Fig.~\ref{fig:lidar_rmse_tau}) for single-photon lidar. The relatively high root mean squared error (RMSE) of ProxiCBO with 60 particles for $\tau$ estimation in Fig.~\ref{fig:lidar_rmse_tau} is due to a few (10 out of 10,000) outliers as we can see in Fig.~\ref{fig:lidar_tau_err}, which explains the high success rate reported in Fig.~\ref{fig:lidar_success_rate}.}
    \label{fig:lidar_rmse}
\end{figure}
\begin{figure}
    \centering
    \begin{subfigure}[b]{0.34\linewidth}
    \includegraphics[width=1.0\linewidth]{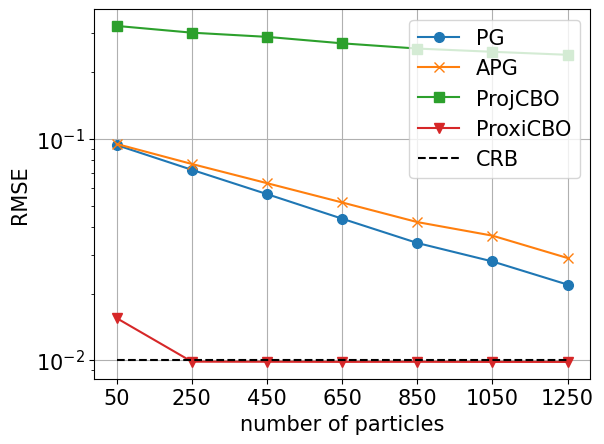}
    \caption{RMSE for $S$ estimation.}
    \label{fig:doppler_lidar_rmse_S}
    \end{subfigure}
    \begin{subfigure}[b]{0.34\linewidth}
    \includegraphics[width=1.0\linewidth]{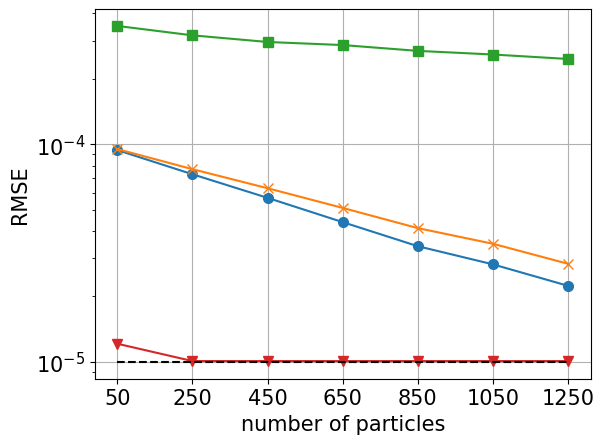}
    \caption{RMSE for $b$ estimation.}
    \label{fig:doppler_lidar_rmse_b}
    \end{subfigure}
    \begin{subfigure}[b]{0.34\linewidth}
    \includegraphics[width=1.0\linewidth]{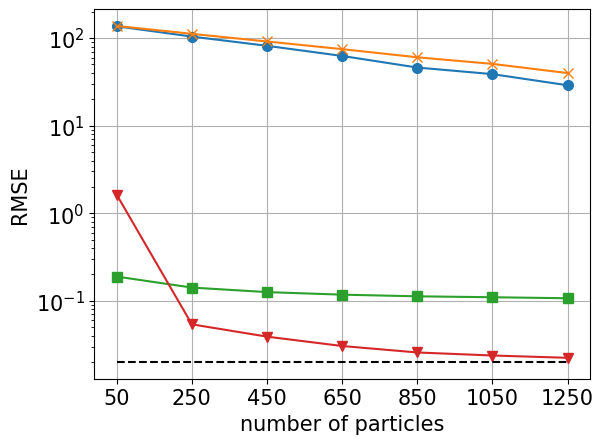}
    \caption{RMSE for $\tau_0$ estimation.}
    \label{fig:doppler_lidar_rmse_tau}
    \end{subfigure}
    \begin{subfigure}[b]{0.34\linewidth}
    \includegraphics[width=1.0\linewidth]{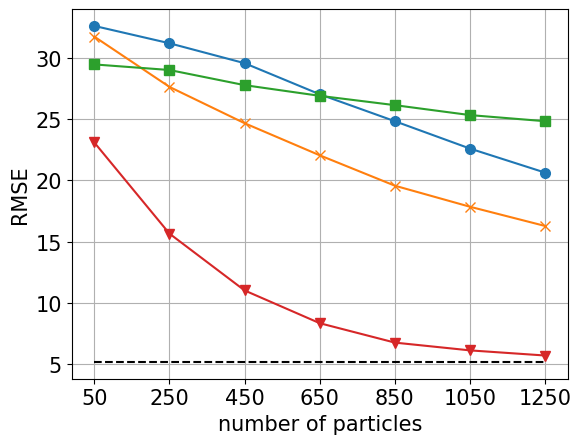}
    \caption{RMSE for $v$ estimation.}
    \label{fig:doppler_lidar_rmse_v}
    \end{subfigure}
    \caption{Estimation of $S$ (Fig.~\ref{fig:doppler_lidar_rmse_S}), $b$ (Fig.~\ref{fig:doppler_lidar_rmse_b}), $\tau_0$ (Fig.~\ref{fig:doppler_lidar_rmse_tau}), and $v$ (Fig.~\ref{fig:doppler_lidar_rmse_v}) for Doppler single-photon lidar. The results are computed from 10,000 independent trials. }
    \label{fig:doppler_lidar_rmse}
\end{figure}
\paragraph{Doppler SPL} We now consider the case where the target is moving with constant velocity $v$. The photon detection process is still an inhomogeneous Poisson process, but the intensity function has changed due to the Doppler-shift effect. Suppose that the target has initial TOF $\tau$ at time $t=0$, then the intensity function for the Poisson process is~\cite{Kitichotkul2025mar, Kitichotkul2025apr}
\begin{equation*}
\lambda(t) = S\sum_{k=1}^K h\left(t-\frac{c \tau}{c-v} - \frac{c+v}{c-v} t_k\right) + b
\end{equation*}
and the log-likelihood function $\mathcal{L}$ is
\begin{equation*}
- SK - b t_a + \sum_{t\in\mathcal{T}} \log\left(S\sum_{k=1}^K h\left(t-\frac{c \tau}{c-v} - \frac{c+v}{c-v}t_k\right) + b\right).
\end{equation*}
Similar to the static case, given detection times $\mathcal{T}$, our goal is to estimate $(S,b,\tau,v)$ by minimizing the negative log-likelihood function under appropriate box constraints for the parameters. Note that in~\cite{Kitichotkul2025mar, Kitichotkul2025apr}, periodic pulse times $\{t_k=k t_r\}_{k=1}^K$ are used, which facilitates an efficient velocity estimation through Fourier probing. In our simulation, we use random pulse times, thus Fourier probing is not applicable. The simulation parameters are the same as the static case, except that we increase the acquisition time $t_a$ and the number of laser pulses $K$ by a factor of 2 for better velocity estimation, while keeping $\text{SBR}=1$. The ground truth velocity is $v_\text{true}=15$ m/s.  For $(S,b,\tau)$, the feasible set and the initial particle distribution are the same as the static case. For velocity $v$, we let the feasible set be $[-50,50]$ m/s and the initial particle distribution be uniform in $[-50,50]$ m/s. Fig.~\ref{fig:doppler_lidar_rmse} shows that proxiCBO outperforms all comparison methods.

\clearpage
\appendix
\section*{Appendix}
In the appendix, we present the detailed proofs of \HZnew{the main theoretical results} in a structured manner. Sections~\ref{app:consensus bound}, \ref{app:stab}, and \ref{app:moment} collect the technical lemmas that will be used throughout the analysis. 
Sections~\ref{app:proof of macro}, \ref{app:proof of quantitative}, \HZnew{\ref{app:proof of long time}, \ref{appen: sc}}, and \ref{app:proof of micro} then provide the proofs of Theorems~\ref{thm:mean-fild existence theorem}, \ref{thm:quantitative-mean-field-limit}, \HZnew{\ref{thm:mean-field long-time-behavior}, Proposition \ref{prop:strongly-convex-main}}, and Theorem \ref{thm:wellposed for micro}, respectively. 
For lemmas whose proofs involve lengthy and technical computations, we defer the details to the supplementary material and present only the core arguments needed for establishing the main theorems in the paper.

\section{Technical Lemmas for Bounding Consensus Point Norms}\label{app:consensus bound}

In the theoretical analysis, the consensus point $v_\alpha(\mu)$~\eqref{eq:def_consensus} associated with a probability distribution $\mu$ arises frequently, and obtaining bounds on its norm is essential. The following proposition provides a quantitative refinement of ~\cite[Proposition A.4]{gerber2023mean}. Whereas the original proof uses a contrapositive argument, we give a direct proof, making the dependence of the constants explicit. In particular, for $p=1$ the result yields a bound on $\|v_\alpha(\mu)\|_2$ in terms of the $L^q(\mu)$ norm of $\|v\|_2$.

\begin{proposition}[{\cite[Proposition A.4]{gerber2023mean}}, quantitative version]\label{A4}
Suppose $\CE\in\CA(s,l)$, and $0<p\leq q$, then there is a constant $C_{\text{Con}}$ such that for all $\mu\in \CP_q(\R^d)$
\begin{gather*}
\frac{\int \|v\|_2^p e^{-\alpha \CE(v)}\,d\mu}{\int e^{-\alpha \CE(v)}\,d\mu} \leq \left(C_{\text{Con}}\cdot \int \|v\|_2^q \,d\mu\right)^{p/q},
\end{gather*}
where $C_{\text{Con}}$ is a constant that only depends on $\alpha,p,q, l,c_u,c_l,C_u,C_l,\underline{\CE}$.
\end{proposition}
\begin{proof}
When $l=0$, the result is straightforward as $e^{-\alpha \CE(v)}$ is both upper and lower bounded. So we are only concerned with the case when $l>0$. By the first part of the proof of \cite[Proposition A.4]{gerber2023mean}, there are two constants $C_1$ and $C_2$ depending only on $\alpha,p,q,l,c_u,c_l,C_u$ and $C_l$ such that
\begin{gather}
\frac{\int \|v\|_2^p e^{-\alpha \CE(v)}\,d\mu}{\int e^{-\alpha \CE(v)}\,d\mu}  \leq  \left(C_1 + C_2 \int \|v\|_2^q \,d\mu\right)^{p/q}.
\end{gather}
Now we consider two cases: $\int \|v\|_2^q \,d\mu> 1$ and $\int \|v\|_2^q \,d\mu\leq 1$. When $\int \|v\|_2^q \,d\mu> 1$, we have
\begin{align*}
&\frac{\int \|v\|_2^p e^{-\alpha \CE(v)}\,d\mu}{\int e^{-\alpha \CE(v)}\,d\mu} \leq  \left(C_1 + C_2 \int \|v\|_2^q \,d\mu\right)^{p/q} \\
&\leq \left((C_1 + C_2) \int \|v\|_2^q \,d\mu\right)^{p/q}.
\end{align*}
For the second case, we consider the set $B_2 (0)$. By Markov inequality, we have $\mu(B^c_2(0)) \leq \int \|v\|_2^q\,d\mu / 2^p \leq 1/2^q$. On the set $B_2 (0)$, we have $\CE(v) \leq c_u 2^l + C_u + \underline{\CE}$. Then we have
\begin{align*}
\int e^{-\alpha \CE(v)}\,d\mu\geq \int_{B_2 (0)} e^{-\alpha \CE(v)}\,d\mu \geq \int_{B_2 (0)} e^{-\alpha(c_u 2^l + C_u+\underline{\CE})}\,d\mu \geq\left(1-\frac{1}{2^q}\right) e^{-\alpha(c_u 2^l + C_u + \underline{\CE})}.
\end{align*}
Thus,
\begin{align*}
\frac{\int \|v\|_2^p e^{-\alpha \CE(v)}\,d\mu}{\int e^{-\alpha \CE(v)}\,d\mu} \leq \frac{e^{-\alpha \underline{\CE}}}{\left(1-\frac{1}{2^q}\right) e^{-\alpha(c_u 2^l + C_u + \underline{\CE})}} \int \|v\|_2^q \,d\mu\leq \frac{e^{-\alpha \underline{\CE}}}{\left(1-\frac{1}{2^q}\right) e^{-\alpha(c_u 2^l + C_u + \underline{\CE})}} \left(\int \|v\|_2^q \,d\mu\right)^{p/q},
\end{align*}
{where the last inequality follows from H\"older inequality.}
Taking $C_{\text{Con}}$ to be the maximum of $C_1+C_2$ and $\left(\tfrac{e^{-\alpha \underline{\CE}}}{\left(1-\frac{1}{2^q}\right) e^{-\alpha(c_u 2^l + C_u + \underline{\CE})}}\right)^{p/q} $ to finish the proof.
\end{proof}
\section{Technical Lemmas for Wasserstein Stability}\label{app:stab}

In proving Theorem~\ref{thm:quantitative-mean-field-limit}, we follow~\cite{gerber2023mean} and adopt Sznitman's argument~\cite[Theorem 3.1]{chaintron2021propagation}, which requires Wasserstein stability of the consensus point, i.e., bounding $\|v_\alpha(\mu)-v_\alpha(\nu)\|_2$ in terms of the Wasserstein distance between $\mu$ and $\nu$. The next proposition provides this estimate and constitutes a quantitative refinement of~\cite[Corollary 3.3]{gerber2023mean}. Whereas the original proof relies on a contrapositive argument, we give a direct proof that makes the constants explicit.

\begin{proposition}[{\cite[Corollary 3.3]{gerber2023mean}}, quantitative version]\label{prop:stab}
Suppose $\CE \in \CA(s,l)$ and $p\geq p_\CM$. Then it holds that for any $(\mu,\nu)\in \CP_{p,R}(\R^d) \times \CP_p(\R^d)$,
\begin{align*}
\|v_{\alpha}(\mu) - v_\alpha(\nu)\|_2\leq C \left(1 + e^{C(1+(2R)^{l} )}\left(1+R^{2p-1}\right)\right) \cdot  \CW_p(\mu,\nu).
\end{align*}
\end{proposition}
\begin{remark}\label{expoenntial_example}
Here, we provide an example that achieves the exponentially scaling coefficient with respect to $R$ in Proposition~\ref{prop:stab}. Let $\CE(v) = v^2$. We know $\CE\in \CA(1,2)$, and $p_\CM = 1$. Pick $p=1\geq p_\CM$ and $\alpha=1$. Let $\mu_n = \delta_n$ and $\nu_n = e^{-n}\delta_0 + (1-e^{-n})\delta_n$. One can verify that the first moments are $ \int v\,d\mu_n = n$ and $\int v\,d\nu_n = n(1-e^{-n})$. One can also compute
\begin{gather*}
v_\alpha (\mu_n) = \frac{\int v e^{-\CE(v)}\,d\mu_n}{\int e^{-\CE (v)}\,d\mu_n}=\frac{n e^{-n^2}}{e^{-n^2}} = n,
\end{gather*}
and
\begin{align*}
v_\alpha(\nu_n) &=  \frac{\int v e^{-\CE(v)}\,d\nu_n}{\int e^{-\CE (v)}\,d\nu_n}=\frac{ne^{-n^2}(1-e^{-n})}{e^{-n} + e^{-n^2} - e^{-n^2 - n}}= n\cdot \frac{1-e^{-n}}{e^{n^2 - n}+1-e^{-n}}.
\end{align*}
Further, $\CW_{1}(\mu_n,\nu_n) = ne^{-n}$. Thus
\begin{gather*}
\frac{|v_\alpha (\mu_n) - v_\alpha (\nu_n)|}{\CW_1 (\mu_n,\nu_n)} = (1 - \frac{1-e^{-n}}{e^{n^2 - n} + 1 - e^{-n}})\cdot e^n.
\end{gather*}
The coefficient scales exponentially with $n$, which is exactly the first moment of $\mu_n$, or $R$.
\end{remark}
The proof of the above proposition relies on the lemma below, which is a quantitative refinement of~\cite[Lemma A.1]{gerber2023mean} where we explicitly track the constants to make their dependence transparent. The details are deferred to Supplementary~\ref{B}.
\begin{lemma}[{\cite[Lemma A.1]{gerber2023mean}}, quantitative version]\label{hglemma}
For a real finite-dimensional vector space $\CV$ with norm $\|\cdot \|$, let $g:\R^d\rightarrow \CV$ and $h:\R^d\rightarrow (0,\infty)$ be functions such that the following condition is satisfied for some $\xi\geq 0$ and $L>0$:
\begin{equation}\label{hg}
\begin{aligned}
\forall (u,v)\in\R^d,\quad\|g(u)-g(v)\| \vee |h(u)-h(v)| \leq L(1+\|u\|+\|v\|)^\xi \|u-v\|.
\end{aligned}
\end{equation}
Let $\eta=1/2 \min_{\|x\|\leq 2^{1/p}R} h(x)$. Then for all $p\geq \xi+1$ and all $R>0$, there exists a constant $C_{p,L}$ only depending on $p$ and $L$ such that for all $(\mu,\nu)\in\CP_{p,R}(\R^d) \times \CP_{p,R}(\R^d)$,
\begin{align*}
\left\| \frac{\int g \,d\mu}{\int h \,d\mu} - \frac{\int g \,d\nu}{\int h \,d\nu}\right\|\leq C_{p,L}\left(\tfrac{1}{\eta} +  \tfrac{\left(\|g(0)\|R + R^p + 1\right)}{\eta^2}\right)\cdot (1+R^{p-1}) \CW_p(\mu,\nu).
\end{align*}
\end{lemma}
Now we are ready to present the proof of Proposition~\ref{prop:stab}
\begin{proof}[Proof of Proposition~\ref{prop:stab}]
First we can verify with $g(v) = v e^{-\alpha \CE(v)}$ and $h(v) = e^{-\alpha \CE(v)}$, the assumptions in Lemma~\ref{hglemma} are satisfied with $L=L_{\alpha,L_\CE,l,c_l,C_l,\underline{\CE}}$ being a constant depending on $\alpha,L_\CE,l,c_l,C_l,\underline{\CE}$, and with $\xi=s+1$ if $l=0$ and $\xi=0$ if $l>0$. Also, one can verify
$\eta=1/2\min_{\|v\|_2\leq 2^{1/p}R }e^{-\alpha \CE(v)}\geq 1/2 e^{-C_{\alpha,l,c_u,C_u}(1+R^{l})}$ and $\|g(0)\|_2=0$. Then for $(\mu,\nu)\in \CP_{p,R}(\R^d) \times \CP_{p,R}(\R^d)$,
\begin{equation}
\begin{aligned}\label{bounded_lipschitz}
&\|v_\alpha (\mu) - v_\alpha (\nu)\|_2\\
&\leq C \cdot e^{C (1+R^l)}\left(1+R^p\right)\left(1+R^{p-1}\right) \CW_p(\mu,\nu)\\
&\leq C\cdot e^{C (1+R^l)} \left(1+R^{2p-1}\right)\CW_p(\mu,\nu),
\end{aligned}
\end{equation}
where the constants do not depend on $R$. Denote $C\cdot e^{C (1+R^l)} \left(1+R^{2p-1}\right)$ by $T(R)$. Also, we know from Proposition~\ref{A4} that for $\mu \in \CP_p(\R^d)$,
\begin{gather}\label{linear-growth}
\|v_\alpha (\mu)\|_2\leq C \CW_p(\mu,\delta_0).
\end{gather}
Given \eqref{bounded_lipschitz} and \eqref{linear-growth}, now we are ready to present the proof. Consider 2 cases: $\nu \in \CP_p(\R^d) \cap \CP_{p,2R}(\R^d)$ or $\nu \in \CP_p(R^d) \cap \CP_{p,2R}^c(\R^d)$. For $\nu \in \CP_p(\R^d) \cap \CP_{p,2R}(\R^d)$, we know from \eqref{bounded_lipschitz}
\begin{gather}\label{case1}
\|v_\alpha (\mu) - v_\alpha (\nu)\|_2\leq T(2R) \CW_p(\mu,\nu).
\end{gather}
For  $\nu \in \CP_p(R^d) \cap \CP_{p,2R}^c(\R^d)$, we know
\begin{equation}\label{case2}
\begin{aligned}
\frac{\|v_\alpha (\mu) - v_\alpha (\nu)\|_2}{\CW_p (\mu,\nu)} &\leq C \cdot \frac{\CW_p(\mu,\delta_0) + \CW_p (\nu,\delta_0)}{\CW_p (\mu,\nu)}\\
&\leq C\cdot \frac{\CW_p(\mu,\delta_0) + \CW_p (\nu,\delta_0)}{\CW_p (\nu,\delta_0) - \CW_p(\mu,\delta_0)}\\
&\leq C\cdot \frac{\CW_p (\nu,\delta_0) + R}{\CW_p (\nu,\delta_0) - R}\\
&\leq 3  C.
\end{aligned}
\end{equation}
Combining \eqref{case1} and \eqref{case2} and summing over the coefficients $T(2R)$ and $3  C$ to finish the proof.
\end{proof}
\section{Technical Lemma on Moment Bounds for \eqref{dyn}}\label{app:moment}
In the theoretical analysis, it is frequently necessary to control the moments of quantities arising from the dynamics~\eqref{dyn}. 
The following proposition summarizes these moment bounds. We defer the computation to Supplementary~\ref{C}.
\begin{proposition}\label{prop:micro_moment_bound}
\HZnew{Suppose $\CE\in\CA(s,l)$ and Assumption~\ref{assump} (iii) holds.} Consider the particle system \eqref{dyn} with initial distribution $\rho_0^{\otimes N}$. Let $\widehat{\rho}^N_t$ be the empirical measure of $V^{1,N}_t,\dots,V^{N,N}_t$. Then for $i=1,\dots,N$,
\begin{equation}\label{three_moment_bounds}
\begin{aligned}
&\E \left[\sup_{t\in[0,T]} \|V_t^{i,N}\|_2^p\right] \vee \E \left[\sup_{t\in[0,T]}\int \|v\|_2^p\,d\widehat{\rho}^N_t\right]\vee \E\left[\sup_{t\in[0,T]} \|v_\alpha (\widehat{\rho}^N_t)\|_2^p\right]\\
&\leq C \left(\E [\|V^{i,N}_0\|_2^p]+k_p(T)\cdot \|v^*\|_2^p\right)\cdot e^{C \cdot T\cdot k_p(T)},
\end{aligned}
\end{equation}
where $k_p(t)$ is defined in Section~\ref{sec:constants}.
\end{proposition}

\section{Proof of Theorem~\ref{thm:mean-fild existence theorem}}\label{app:proof of macro}
The proof is based on the Leray-Schauder Theorem below.
\begin{theorem}[{\cite[Theorem~11.3]{GT}}]
\label{schroader}
	Let $ \CT $ be a compact mapping of a Banach space $ \mathcal{B} $ into itself, and suppose there exists a constant $ M $ such that $\left\|x\right\|_{\mathcal{B}}<M $ for all $ x\in \CB $ and $ \sigma\in[0,1] $ satisfying $ x=\sigma \CT x $. Then $ \CT $ has a fixed point.
\end{theorem}
The key idea of the proof is to choose a suitable space $\mathcal{B}$ and an appropriate mapping $\mathcal{T}$. 
Before presenting the proof, we first provide an auxiliary lemma that can be used to establish the well-definedness of the chosen operator $\mathcal{T}$. We defer the computation to Supplementary~\ref{D}.
\begin{lemma}\label{lemma:well-definedness of T}
Suppose $\CE\in \CA(s,l)$ with parameters $s,l\geq 0$, such that $l\leq s+1$ \HZnew{and suppose Assumption~\ref{assump} (iii) holds.} Fix a final time $T$. Assume also $\rho_0 \in \CP_{p}(\R^d)$ for $p\geq \max\{2,p_{\CM} (s,l) \}$, and let $\overline{V}_0 \sim \rho_0$. Given any $u \in \mathcal{C}([0,T],\R^d)$, the SDE
\begin{equation}\label{SDE_general}
\begin{aligned}
    d \overline{V}_t =& -\lambda_1 \left(\overline{V}_t - u_t\right)\,dt \\
    &-\lambda_2 \left( \nabla f(\overline{V}_t) + \nabla M_{\mu g} \left(\overline{V}_t - \mu \nabla f(\overline{V}_t)\right)\right)\,dt\\
    & + \sigma_1 D\left(\overline{V}_t - u_t\right) \, dB_t\\
    & + \sigma_2 D\left( \nabla f(\overline{V}_t) + \nabla M_{\mu g} \left({V}_t - \mu \nabla f(\overline{V}_t)\right)\right)\, d\widetilde{B}_t,
\end{aligned}
\end{equation}
has a unique almost surely continuous strong solution $\overline{V}$. Moreover, let $\rho_t$ be the law of $\overline{V}_t$, then the function $t\rightarrow v_\alpha (\rho_t)$ belongs to $\mathcal{C} ([0,T],\R^d)$. Further, one has
\begin{equation}\label{vs_p_bound}
\begin{aligned}
\E \left[\sup_{s\in[0,T]} \|\overline{V}_s\|_2^p\right]\leq C \left(\E \|\overline{V}_0\|_2^p  +  k_p(T)\left(\|v^*\|_2^p+ \|u\|^p_{L^\infty([0,T])}\right)\right) e^{C\cdot T\cdot k_p(T)}
\end{aligned}
\end{equation}
where $k_p(t)$ is defined in Section~\ref{sec:constants}.
\end{lemma}
Now we are ready for the proof of Theorem~\ref{thm:mean-fild existence theorem}.
\begin{proof}[Proof of Theorem~\ref{thm:mean-fild existence theorem}]
We use Theorem~\ref{schroader} to do the proof. 
\paragraph{\textbf{Definition of $\CB$ and $\CT$}}
Let $\CB = \mathcal{C}([0, T], \R^d)$ equipped with $L^\infty ([0,T])$ norm. For any $u \in \mathcal{C}([0,T],\R^d)$, by Lemma~\ref{lemma:well-definedness of T}, there is a unique {$\rho_t^{(u)}$} determined by SDE \eqref{SDE_general}. We define $\CT(u)(\cdot):= v_\alpha ({\rho_{\cdot}^{(u)}})$. Again by Lemma~\ref{lemma:well-definedness of T}, $\CT (u) \in\mathcal{C} ([0,T], \R^d)$. Thus $\CT$ is a well-defined map from $\mathcal{C} ([0,T],\R^d)$ to itself. Suppose we can show that $\CT$ has a fixed point $u^*$, then we have $u^*_{\cdot}=\CT(u^*)(\cdot):=v_\alpha(\rho_{\cdot}^{(u^*)})$. Setting $u_t=u_t^*=v_\alpha(\rho_t^{(u^*)})$ in \eqref{SDE_general}, by Lemma~\ref{lemma:well-definedness of T} and our definition of $\CT$, we have that $\overline{V}_t\sim\rho_t^{(u^*)}$ is the unique strong solution to \eqref{SDE_general}. Note that \eqref{SDE_general} with $u_t=v_\alpha(\text{Law}(\overline{V}_t))$ is the same as \eqref{SDE}, hence this proves that there exists a strong solution to \eqref{SDE}.

In the following, we will show that $\CT$ defined above has a fixed point by showing that $\CT$ satisfies the conditions in Theorem~\ref{schroader}.
\paragraph{\textbf{Verify $\CT$ is compact}} Given any bounded set $S$ in $\CB$, where $S=\{u \mid u\in\mathcal{C} ([0,T],\R^d), \|u\|_{L^\infty([0,T])} \leq R\}$, the goal is to show $\overline{\CT(S)}$ is compact. By Arzel\`a-Ascoli theorem, it suffices to show $\CT (S)$ is pointwise-bounded and equicontinuous.
Given $u\in S$ and $0\leq r\leq s\leq T$. Let $\rho_t$ be the law of the solution $\overline{V}_t$ in \eqref{SDE_general} determined by $u$. We know {from Lemma~\ref{lemma:well-definedness of T} that} for all $t\in[0,T]$, $\rho_t\in\CP_{p, K}$, with $K \lesssim \left(1+\int \|v\|_2^p\,d\rho_0 + \|u\|^p_{L^\infty([0,T])}\right){\lesssim} \left(1+\int \|v\|_2^p\,d\rho_0 + R^p\right)$. Then, we have
\begin{align*}
&\|v_\alpha (\rho_r) - v_\alpha(\rho_s)\|_2\\
&\overset{(a)}{\lesssim}  \CW_p(\rho_r,\rho_s)\lesssim(\E [\|V_s - V_r\|_2^p ])^{1/p}\\
&\overset{(b)}{\lesssim}\Big(\left[(s - r)^{p-1} + (r-s)^{\tfrac{p}{2} - 1}\right] \times \int_r^s \left(1 + \|u_s\|_2^p + \E\|\overline{V}_s\|_2^p\right) \,ds\Big)^{1/p}\\
&\overset{(c)}{\lesssim} \Big(\left[(s-r)^{p-1} + (s-r)^{\tfrac{p}{2} - 1}\right]\times \E \left[\int_r^s \left(1 {+ \E\|\overline{V}_0\|_2^p} + \|u\|^p_{L^\infty([0,T])}\right) \,ds\right]\Big)^{1/p}\\
&\lesssim (s-r)^{1/2},
\end{align*}
where step $(a)$ is due to Proposition~\ref{prop:stab}, step $(b)$ is due to (\ref{holder_pre}) in Supplementary~\ref{D}, and step $(c)$ is due to {\eqref{vs_p_bound}}. Here the constant does not depend on $r,s$ and $u$ as $\|u\|_{L^\infty ([0,T])}$ is bounded by $R$. Then for $0\leq r\leq s\leq T$, and $u\in S$, one has
\begin{gather}\label{Holder}
\|\CT(u)(r) - \CT(u)(s)\|_2 = \|v_\alpha (\rho_r) - v_\alpha(\rho_s)\|_2\lesssim (s-r)^{1/2}.
\end{gather}
This gives equicontinuity. For pointwise-boundedness, for any $u\in S$, by \eqref{Holder} and Proposition~\ref{A4}, one has
\begin{align*}
&\|\CT(u)(r)\|_2 \leq \|\CT(u)(0)\|_2 + \|\CT(u)(r) - \CT(u)(0)\|_2\\
&\leq \|v_\alpha(\rho_0)\|_2 + CT^{1/2} \lesssim  (\int \|v\|_2^p \,d\rho_0)^{1/p} + \sqrt{T}.
\end{align*}
Having obtained equicontinuity and pointwise-boundedness,  Arzel\`a-Ascoli theorem implies $\CT$ is compact.
\paragraph{\textbf{Show $\mathcal{U}:=\{u\mid u\in\CB, \sigma\CT(u)= u \text{ for some $\sigma\in[0,1]$}\}$ is bounded}}
{We will show that $\mathcal{U}$ is bounded by showing that its elements are uniformly bounded. For any $u\in\mathcal{U}$,} denoting $T^{p-1}+T^{\tfrac{p}{2} - 1}$ by $P(T)$, one has
\begin{align*}
&\E \left[\sup_{s\in[0,t]} \|\overline{V}_s\|_2^p\right]\\
&\overset{(a)}{\leq} C \Big( \E \|\overline{V}_0\|_2^p+  P(T)\cdot\E\left[\int_0^t \left(  \|v^*\|_2^p+ \|u_s\|_2^p + \|\overline{V}_s\|_2^p\right) \,ds\right]\Big)\\
&\overset{(b)}{=} C \Big( \E \|\overline{V}_0\|_2^p+  P(T)\cdot \E\left[\int_0^t   \left(\|v^*\|_2^p+ \|\sigma v_\alpha (\rho_s)\|_2^p + \|\overline{V}_s\|_2^p\right) \,ds\right]\Big)\\
&\overset{(c)}{\leq} C \Big( \E \|\overline{V}_0\|_2^p+  P(T)\cdot\E\left[\int_0^t   \left(\|v^*\|_2^p + \|\overline{V}_s\|_2^p\right) \,ds\right]\Big)\\
&\leq C  \Big(\E \left[ \|\overline{V}_0\|_2^p \right]+k_p(T)\|v^*\|_2^p+k_p(T)\cdot \E \left[\int_0^t \sup_{r\in[0,s]} \|\overline{V}_r\|_2^p \,ds\right]\Big),
\end{align*}
where step $(a)$ follows from the first inequality in (\ref{holder_mid}) in Supplementary~\ref{D}, step $(b)$ holds because $u=\sigma \CT(u)$ for $u\in\mathcal{U}$, and step $(c)$ uses Proposition~\ref{A4} for bounding $\|v_\alpha(\rho_s)\|_2^p$. Then applying the Gr\"onwall's inequality gives 
\begin{gather*}
\E \left[\sup_{s\in[0,t]} \|\overline{V}_s\|_2^p\right]\leq C \left(\E \left[ \|\overline{V}_0\|_2^p \right]+k_p(T)\|v^*\|_2^p\right)e^{C\cdot T\cdot k_p(T)}.
\end{gather*}
Thus by Proposition~\ref{A4} again, one has
\begin{align*}
\|u\|_{L^\infty([0,T])} &= \sigma \|\CT(u)\|_{L^\infty([0,T])} =\sigma \|v_\alpha(\rho_t)\|_{L^\infty([0,T])}\\
&\leq C \sup_{s\in[0,T]}(\E[ \|\overline{V}_s\|_2^p])^{1/p} \\
&\leq  C\left(\E \left[ \|\overline{V}_0\|_2^p \right]+k_p(T)\|v^*\|_2^p\right)^{1/p}e^{C\cdot T\cdot k_p(T)}.
\end{align*}
This proves {that all $u\in\mathcal{U}$ are uniformly bounded, which implies} the boundedness {of $\mathcal{U}$}. Then by Theorem~\ref{schroader}, the existence is established, as well as the bound in \eqref{mean-field-bound}.
\paragraph{\textbf{Uniqueness and weak solution}} 
This part can be established using well-developed techniques.
We refer the readers to, for example,~\cite[Theorem 2.4]{gerber2023mean},~\cite[Theorem 3.1]{cbo2}.
\end{proof}
\section{Proof of Theorem~\ref{thm:quantitative-mean-field-limit}}\label{app:proof of quantitative}

We now present the proof of Theorem~\ref{thm:quantitative-mean-field-limit}. 
As noted earlier, the argument follows the approach of Sznitman~\cite[Theorem 3.1]{chaintron2021propagation}. 
A key step is to control $
\E\!\left[\|v_\alpha(\widehat{\rho}^N_s) - v_\alpha(\rho_s)\|_2^2\right],$ where $\widehat{\rho}^N_s$ denotes the empirical measure of the $N$-particle system generated by~\eqref{dyn}, and $\rho_s$ is the law of the mean-field dynamics~\eqref{SDE}. 
This quantity can be bounded by
\[2 \Big( \E\left[\|v_\alpha(\widehat{\rho}^N_s) - v_\alpha(\widehat{\overline{\rho}}^N_s)\|_2^2\right] 
+ \E\left[\|v_\alpha(\widehat{\overline{\rho}}^N_s) - v_\alpha(\rho_s)\|_2^2\right] \Big),\]
where $\widehat{\overline{\rho}}^N_s$ is the empirical measure of $N$ i.i.d.~copies of the mean-field particle~\eqref{SDE}. 

For the first term, we apply the stability estimate in Proposition~\ref{prop:stab}, together with the consensus bound in Proposition~\ref{A4}, the moment bounds in Proposition~\ref{prop:micro_moment_bound}, and Theorem~\ref{thm:mean-fild existence theorem}, to obtain the following lemma.

\begin{lemma}\label{lem:first_term}
Let $\CE\in\CA(s,l)$ with $0<l\leq s+1$ or $\CE\in \CA(0,0)$, \HZnew{and suppose Assumption~\ref{assump} (iii) holds.} Assume $V_0\sim\rho_0 \in \CP (\R^d)$ has bounded moments of all orders. Moreover, let $\{V_t^{i,N}\}_{i=1}^N$ be the solution to \eqref{dyn} with
\HZnew{$V_0^{i,N}\overset{\text{i.i.d.}}{\sim}\rho_0$, $i=1,\ldots,N$}, and let
\HZnew{$\{\overline{V}_t^{i,N}\}_{i=1}^N$ be copies of the solution
to \eqref{SDE} synchronously coupled with $\{V_t^{i,N}\}_{i=1}^N$
as in Theorem~\ref{thm:quantitative-mean-field-limit}.} Consider empirical distributions $
\widehat{\rho}_s^N = \sum_{i=1}^N \delta_{V^{i,N}_s}/N$ and $
\widehat{\overline{\rho}}_s^N = \sum_{i=1}^N \delta_{\overline{V}^{i,N}_s}/N$. Then 
\begin{align*}
&\E\left[\left\|v_\alpha (\widehat{\rho}^N_s) - v_\alpha (\widehat{\overline{\rho}}^N_s)\right\|_2^2\right]\\
&\leq C\cdot\left(\E \left[ \|V_0\|_2^8 \right]+k_8(T)\|v^*\|_2^8\right)^{3/4}\cdot e^{C\cdot T\cdot k_8(T)}\cdot N^{-1} +C\cdot \Psi_{\rho_0,T,v^*} \cdot \E\left[\left\|V^{1,N}_s - \overline{V}^{1,N}_s\right\|_2^2\right],
\end{align*}
where $\Psi_{\rho_0,T,v^*}$ is defined in Section~\ref{sec:constants}.
\end{lemma}
For the second term, we rely on a result from importance sampling. 
Specifically, the following theorem is the vector-valued analogue of~\cite[Theorem~2.3]{agapiou2017importance}: 
while the original statement applies to scalar functions $\phi$, here we extend it to vector-valued functions.
\begin{theorem}\HZnew{({\cite[Theorem~2.3; vector version]{agapiou2017importance}})}\label{thm:MSEbound}
Let $\mu\in\CP(\R^d)$. Consider functions $\phi:\R^d\rightarrow \R^d$ and $g: \R^d\rightarrow \R_{++}$.
Let $\phi_k$ be the function defined by the $k$-th entry of $\phi$. Suppose the below $C_{\text{MSE}}$ is finite,
\begin{align*}
C_{\text{MSE}}:=&\frac{3}{(\int g\,d\mu)^2} \sum_{k=1}^d\CM_2 (\phi_k g)\\
&+ \frac{3 C_{2m}^{1/m} \CM_{2m}(g)^{1/m}}{(\int g \,d\mu)^4} \cdot \left(\int \left(\|\phi\|_2g\right)^{2l}\,d\mu\right)^{1/l}\\
&+\frac{3C^{1/q}_{2q(1+\frac{1}{p})} \CM_{2q(1+\frac{1}{p})} (g)^{1/q}}{(\int g\,d\mu)^{2(1+\frac{1}{p})}} \left(\int \|\phi\|_2^{2p}\,d\mu\right)^{1/p}.
\end{align*}
Then one has
\begin{gather*}
\E\left[\left\|\frac{\int \phi g \,d\widehat{\mu}^N}{\int g \,d\widehat{\mu}^N} - \frac{\int \phi g \,d\mu}{\int g \,d\mu}\right\|_2^2\right] \leq \frac{1}{N}C_{\text{MSE}}.
\end{gather*}
Here, $\widehat{\mu}^N=\sum_{j=1}^N\delta_{u^{j,N}}/N$ is the empirical measure of $u^{1,N},\dots,u^{N,N}$, which are sampled independently from $\mu$. The constants $C_t$ satisfies $C_t^{1/t}\leq t-1$ and the two pairs of parameters $l,m$ and $p,q$ are conjugate to each other, i.e., $\frac{1}{l}+\frac{1}{m}=1$ and $\frac{1}{p}+\frac{1}{q}=1$.  $\CM_t (f)$ is the $t$-th central moment of $f:\R^d \rightarrow \R$ under the distribution $\mu$. 
\end{theorem}
Using this theorem, we control the second term as follows.
\begin{lemma}\label{thm:importance sampling}
Let $\CE\in\CA(s,l)$ with $0<l\leq s+1$ or $\CE\in\CA(0, 0)$. \HZnew{Suppose Assumption~\ref{assump} (iii) holds.} Let $\overline{V}_s$, whose law is denoted by $\rho_s$, be the solution of \eqref{SDE} given by Theorem~\ref{thm:mean-fild existence theorem}, with initial condition $\rho_0$ that admits bounded moments of all orders. Consider $N$ independent copies of $\overline{V}_s$, denoted by $\left\{\overline{V}^{j,N}_s\right\}_{j=1}^N$. Let the empirical distribution $
\widehat{\overline{\rho}}_s^N = \sum_{j=1}^N \delta_{\overline{V}^{j,N}_s}/N$. Then 
\begin{align*}
\E \left[\left\|v_\alpha (\widehat{\overline{\rho}}_s^N) - v_\alpha (\rho_s)\right\|_2^2\right]\leq C \cdot \frac{e^{\Phi_{\rho_0, T,v^*}}\left(\max\{k_4(T),k_2^2(T)\}\|v^*\|_2^4 + \E[\|\overline{V}_0\|_2^4]\right)^{1/2}}{N},
\end{align*}
where $\Phi_{\rho_0,T,v^*}$ is defined in Section~\ref{sec:constants}.
\end{lemma}
The proofs of Lemma~\ref{lem:first_term}, Theorem~\ref{thm:MSEbound}, and Lemma~\ref{thm:importance sampling} are deferred to Supplementary~\ref{E}. Using those two lemmas, we are ready to present the proof of Theorem~\ref{thm:quantitative-mean-field-limit}.
\begin{proof}[Proof of Theorem~\ref{thm:quantitative-mean-field-limit}]
We use the following notation:
\begin{align*}
&a(V,\rho) = -\lambda_1 (V - v_\alpha (\rho))-\lambda_2 \left(\nabla f(V) + \nabla M_{\mu g} \left(V - \mu \nabla f(V)\right)\right),\\
&b_1 (V,\rho) = \sigma_1 D\left(V - v_\alpha (\rho)\right),\\
&b_2 (V) = \sigma_2 D\left(\nabla f(V) + \nabla M_{\mu g}\left(V - \mu \nabla f(V)\right)\right).
\end{align*}
We start by Burkholder-Davis-Gundy (BDG) inequality~\cite[Theorem 7.3]{mao2007stochastic} to obtain, 
\begin{equation}\label{mean field limit main inequality}
\begin{aligned}
&&&\E \left[\sup_{s\in[0,t]} \left\|V_s^{i,N} - \overline{V}_s^{i,N}\right\|_2^2\right]\\
&\overset{(a)}{\leq}&&  2 \E \left(\int_0^t \left\| a\left(V^{i,N}_s,\widehat{\rho}^{N}_s\right) - a\left(\overline{V}^{i,N}_s,\rho_s\right)\right\|_2\,ds\right)^2\\
&&& + 2C_{\rm BDG} \E\Bigg(\int_0^t \left\| b_1\left(V^{i,N}_s, \widehat{\rho}^{N}_s\right) - b_1\left(\overline{V}^{i,N}_s,\rho_s\right)\right\|_F^2 + \left\| b_2\left(V^{i,N}_s\right) - b_2\left(\overline{V}^{i,N}_s\right)\right\|_F^2\,ds\Bigg)\\
&\overset{(b)}{\leq}&& 2 \cdot T\cdot \E \left(\int_0^t \left\| a\left(V^{i,N}_s,\widehat{\rho}^{N}_s\right) - a\left(\overline{V}^{i,N}_s,\rho_s\right)\right\|_2^2\,ds\right)\\
&&& + 2C_{\rm BDG} \E\Bigg(\int_0^t \left\| b_1\left(V^{i,N}_s, \widehat{\rho}^{N}_s\right) - b_1\left(\overline{V}^{i,N}_s, \rho_s\right)\right\|_F^2 + \left\| b_2\left(V^{i,N}_s\right) - b_2\left(\overline{V}^{i,N}_s\right)\right\|_F^2\,ds\Bigg)\\
&\overset{(c)}{\leq}&&C\left(1+T\right)\Bigg(\int_0^t \E \left[\left\|V_s^{i,N} - \overline{V}_s^{i,N}\right\|_2^2\right] + \E\left[\left\|v_\alpha (\widehat{\rho}^N_s) - v_\alpha (\rho_s)\right\|_2^2\right]\,ds\Bigg),
\end{aligned}
\end{equation}
where step (a) is derived from BDG inequality, step $(b)$ follows from Cauchy-Schwarz inequality and step $(c)$ holds since $\nabla f$, $\nabla M_{\mu g}$ and $D$ are Lipschitz.

Thus to apply \HZnew{Gr\"onwall's} inequality, it suffices to bound $\E\left[\left\|v_\alpha (\widehat{\rho}^N_s) - v_\alpha (\rho_s)\right\|_2^2\right]$. Notice it is bounded by
\begin{align*}
 2\left(\E\left[\left\|v_\alpha (\widehat{\rho}^N_s) - v_\alpha (\widehat{\overline{\rho}}^N_s)\right\|_2^2\right]+ \E\left[\left\|v_\alpha (\widehat{\overline{\rho}}^N_s) - v_\alpha (\rho_s)\right\|_2^2\right]\right).
\end{align*}
For the first term, by Lemma~\ref{lem:first_term}, one has
\begin{align*}
&\E\left[\left\|v_\alpha (\widehat{\rho}^N_s) - v_\alpha (\widehat{\overline{\rho}}^N_s)\right\|_2^2\right] \\
&\leq C\cdot\left(\E \left[ \|V_0\|_2^8 \right]+k_8(T)\|v^*\|_2^8\right)^{3/4}\cdot e^{C\cdot T\cdot k_8(T)}\cdot N^{-1} +C\cdot \Psi_{\rho_0,T,v^*} \cdot \E\left[\left\|V^{1,N}_s - \overline{V}^{1,N}_s\right\|_2^2\right],
\end{align*}
where $\Psi_{\rho_0,T,v^*}$ is defined in Section~\ref{sec:constants}.
For the second term, from Lemma~\ref{thm:importance sampling}, one has
\begin{align*}
\E\left[\left\|v_\alpha (\widehat{\overline{\rho}}^N_s) - v_\alpha (\rho_s)\right\|_2^2\right]\leq \frac{C \cdot e^{\Phi_{\rho_0, T,v^*}}\left(\max\{k_4(T),k_2^2(T)\}\|v^*\|_2^4 + \E[\|V_0\|_2^4]\right)^{1/2}}{N},
\end{align*}
where $\Phi_{\rho_0,T,v^*}$ is defined in Section~\ref{sec:constants}.
Then combining the above two inequalities, one has
\begin{align*}
\E\left[\left\|v_\alpha (\widehat{\rho}^N_s) - v_\alpha (\rho_s)\right\|_2^2\right] \leq C \cdot \Lambda_{\rho_0,T,v^*} \cdot N^{-1} + C\cdot \Psi_{\rho_0,T,v^*} \cdot \E\left[\left\|V^{1,N}_s - \overline{V}^{1,N}_s\right\|_2^2\right],
\end{align*}
where $\Lambda_{\rho_0,T,v^*}$ is defined in Section~\ref{sec:constants}.
Plugging the above inequality into \eqref{mean field limit main inequality}, one obtains
\begin{align*}
&&&\E \left[\sup_{s\in[0,t]} \left\|V_s^{i,N} - \overline{V}_s^{i,N}\right\|_2^2\right]
\\
&\leq&&  C \cdot \left(T+T^{2}\right)\cdot\Lambda_{\rho_0,T,v^*}\cdot N^{-1}\\
&&&+ C\cdot\left(1+T\right)\cdot \left(1+\Psi_{\rho_0,T,v^*}\right)\times \left(\int_0^t\E\left[\left\|V^{1,N}_s - \overline{V}^{1,N}_s\right\|_2^2\right]\,ds \right)\\
&\leq&&  C \cdot \left(T+T^{2}\right)\cdot\Lambda_{\rho_0,T,v^*}\cdot N^{-1}\\
&&&+ C\cdot\left(1+T\right)\cdot \left(1+\Psi_{\rho_0,T,v^*}\right)\times\left(\int_0^t  \sup_{r\in [0, s]}\E\left[\left\|V^{1,N}_r - \overline{V}^{1,N}_r\right\|_2^2\right]\,ds \right).
\end{align*}
Then one can use Gr\"onwall inequality to obtain
\begin{align*}
\E \left[\sup_{s\in[0,T]} \left\|V_s^{i,N} - \overline{V}_s^{i,N}\right\|_2^2\right] \leq  C \left(T+T^{2}\right)\cdot e^{ C\cdot\left(T+T^{2}\right)\cdot \left(1+\Psi_{\rho_0,T,v^*}\right)} \cdot \Lambda_{\rho_0,T,v^*}\cdot N^{-1}.
\end{align*}
This completes the proof.
\end{proof}
\HZnew{\section{Proof of Theorem~\ref{thm:mean-field long-time-behavior}}\label{app:proof of long time}}
\begin{proof}
\HZnew{The proof follows the approach of~\cite[Corollary~2.6]{memory}. We provide the main steps and verify how the additional proximal-gradient terms in~\eqref{SDE} enter the argument.}

\HZnew{To simplify the notation, let $
G_\mu(v):=\nabla f(v)+\nabla M_{\mu g}\left(v-\mu\nabla f(v)\right).$ By the relation between the Moreau envelope and the proximal operator in~\eqref{eq:moreau_prox}, $
G_\mu(v)
=
\frac{1}{\mu}
\left(
v-P_\mu(v)
\right)
$, where $
P_\mu(v)=
\operatorname{prox}_{\mu g}
\left(v-\mu\nabla f(v)\right)$.
Since $v^*$ minimizes $\CE=f+g$, one has
\begin{gather}\label{eq:G_mu_at_vstar}
P_\mu(v^*)=v^*\quad\text{and}\quad G_\mu(v^*)=0.
\end{gather}
Furthermore, by the $1/\mu$-Lipschitz property of the Moreau envelope gradient and Assumption~\ref{assump} (iii),
\begin{align}
\|G_\mu(v)-G_\mu(w)\|_2
&\leq
\left(2L_f+\frac{1}{\mu}\right)\|v-w\|_2.
\label{eq:G_mu_lip}
\end{align}
We denote $L_\mu:=2L_f+\frac{1}{\mu}$.} \HZnew{Consider the energy functional $
\CV(\rho_t)
:=
\frac{1}{2}\CW_2^2(\rho_t,\delta_{v^*})
=
\frac{1}{2}\int\|v-v^*\|_2^2\,d\rho_t(v).$ Using the weak formulation of~\eqref{FPK} (similar as~\cite[Remark~2.3]{memory}, $\tfrac{1}{2}\|v - v^*\|_2^2$ can be chosen as a test function) gives
\begin{equation}\label{eq:mean_field_V_evolution}
\begin{aligned}
\frac{d}{dt}\CV(\rho_t)=&
-\lambda_1
\int
\left\langle
v-v^*,v-v_\alpha(\rho_t)
\right\rangle
\,d\rho_t(v)\\
&-\lambda_2
\int
\left\langle
v-v^*,G_\mu(v)
\right\rangle
\,d\rho_t(v)+\frac{\sigma_1^2}{2}
\int
\|v-v_\alpha(\rho_t)\|_2^2\,d\rho_t(v)\\
&+\frac{\sigma_2^2}{2}
\int
\|G_\mu(v)\|_2^2\,d\rho_t(v).
\end{aligned}
\end{equation}
 Since $P_\mu(v^*)=v^*$~\eqref{eq:G_mu_at_vstar} and $
\|P_\mu(v)-P_\mu(v^*)\|_2
\leq
(1+\mu L_f)\|v-v^*\|_2$, one has 
\begin{equation}\label{eq:G_mu_inner}
\begin{aligned}
\left\langle v-v^*,G_\mu(v)\right\rangle
&=
\frac{1}{\mu}
\left(
\|v-v^*\|_2^2
-
\left\langle
v-v^*,P_\mu(v)-v^*
\right\rangle
\right)\geq
-L_f\|v-v^*\|_2^2.
\end{aligned}
\end{equation}
We now bound each term in~\eqref{eq:mean_field_V_evolution}. First, by writing $
v-v_\alpha(\rho_t)=(v-v^*)-(v_\alpha(\rho_t)-v^*),$ we obtain
\begin{align*}
&-\lambda_1
\int
\left\langle
v-v^*,v-v_\alpha(\rho_t)
\right\rangle
\,d\rho_t(v)+
\frac{\sigma_1^2}{2}
\int
\|v-v_\alpha(\rho_t)\|_2^2\,d\rho_t(v)
\\
&=
-\left(2\lambda_1-\sigma_1^2\right)\CV(\rho_t)+
(\lambda_1-\sigma_1^2)
\left\langle
\int(v-v^*)\,d\rho_t(v),
v_\alpha(\rho_t)-v^*
\right\rangle
+
\frac{\sigma_1^2}{2}
\|v_\alpha(\rho_t)-v^*\|_2^2.
\end{align*}
Since $
\left\|
\int(v-v^*)\,d\rho_t(v)
\right\|_2
\leq
\int\|v-v^*\|_2\,d\rho_t(v)
\leq
\sqrt{2\CV(\rho_t)},$ it follows that
\begin{equation}\label{eq:consensus_V_bound}
\begin{aligned}
&-\lambda_1
\int
\left\langle
v-v^*,v-v_\alpha(\rho_t)
\right\rangle
\,d\rho_t(v)+
\frac{\sigma_1^2}{2}
\int
\|v-v_\alpha(\rho_t)\|_2^2\,d\rho_t(v)\\
&\leq
-\left(2\lambda_1-\sigma_1^2\right)\CV(\rho_t)
+
\sqrt{2}(\lambda_1+\sigma_1^2)
\sqrt{\CV(\rho_t)}
\|v_\alpha(\rho_t)-v^*\|_2+
\frac{\sigma_1^2}{2}
\|v_\alpha(\rho_t)-v^*\|_2^2.
\end{aligned}
\end{equation}
For the proximal-gradient terms, by~\eqref{eq:G_mu_inner} and
\eqref{eq:G_mu_at_vstar}--\eqref{eq:G_mu_lip}, 
\begin{align}
&-\lambda_2
\int
\left\langle
v-v^*,G_\mu(v)
\right\rangle
\,d\rho_t(v)
+
\frac{\sigma_2^2}{2}
\int
\|G_\mu(v)\|_2^2\,d\rho_t(v)
\nonumber\\
&\leq
\lambda_2L_f
\int\|v-v^*\|_2^2\,d\rho_t(v)
+
\frac{\sigma_2^2L_\mu^2}{2}
\int\|v-v^*\|_2^2\,d\rho_t(v)
\nonumber\\
&=
\left(
2\lambda_2L_f+\sigma_2^2L_\mu^2
\right)\CV(\rho_t).
\label{eq:prox_V_bound}
\end{align}
Combining~\eqref{eq:consensus_V_bound} and~\eqref{eq:prox_V_bound}
therefore gives
\begin{align*}
\frac{d}{dt}\CV(\rho_t)
\leq{}&
-\left(
2\lambda_1-\sigma_1^2
-2\lambda_2L_f
-\sigma_2^2L_\mu^2
\right)\CV(\rho_t)
\\
&+
\sqrt{2}(\lambda_1+\sigma_1^2)
\sqrt{\CV(\rho_t)}
\|v_\alpha(\rho_t)-v^*\|_2+
\frac{\sigma_1^2}{2}
\|v_\alpha(\rho_t)-v^*\|_2^2.
\end{align*}
Finally, since $2L_f\leq L_\mu$, where
$L_\mu=2L_f+\frac{1}{\mu}$, we obtain 
\begin{equation}\label{eq:mean_field_V_upper}
\begin{aligned}
\frac{d}{dt}\CV(\rho_t)
\leq{}&
-\chi\CV(\rho_t)
+\sqrt{2}(\lambda_1+\sigma_1^2)
\sqrt{\CV(\rho_t)}
\|v_\alpha(\rho_t)-v^*\|_2+
\frac{\sigma_1^2}{2}
\|v_\alpha(\rho_t)-v^*\|_2^2,
\end{aligned}
\end{equation}
where $
\chi
:=
2\lambda_1-\sigma_1^2
-\lambda_2L_\mu-\sigma_2^2L_\mu^2>0.$}

\HZnew{It therefore remains to control the distance between the consensus point
$v_\alpha(\rho_t)$ and $v^*$. Let $c_\chi>0$ be chosen that
\begin{gather}\label{eq:c_chi}
\sqrt{2}(\lambda_1+\sigma_1^2)c_\chi
+\frac{\sigma_1^2}{2}c_\chi^2
\leq
\frac{\chi}{2}.
\end{gather}
Then, whenever
\begin{gather}\label{eq:consensus_control}
\|v_\alpha(\rho_t)-v^*\|_2
\leq
c_\chi\sqrt{\CV(\rho_t)},
\end{gather}
combining~\eqref{eq:mean_field_V_upper} and~\eqref{eq:c_chi} gives
\begin{gather}\label{eq:mean_field_decay}
\frac{d}{dt}\CV(\rho_t)
\leq
-\frac{\chi}{2}\CV(\rho_t).
\end{gather}}

\YMnew{In the following, we show that we can choose a sufficiently large $\alpha$ such that~\eqref{eq:consensus_control} holds until $\CV(\rho_t)<\delta/2$ is reached.}

\HZnew{
For
$r>0$, let $
\CE_r
:=
\sup_{v\in B_r^\infty(v^*)}
\left(\CE(v)-\CE(v^*)\right).$ Define
\begin{gather}\label{eq:q_delta}
q_\delta
:=
\frac{1}{2}
\min\left\{
\left(
\frac{\eta c_\chi\sqrt{\delta/2}}
{2\sqrt{d}}
\right)^{1/\nu},
\CE_\infty
\right\}>0.
\end{gather}
By the continuity of $\CE$ at $v^*$, we can choose
$r_\delta\in(0,R_0]$ sufficiently small such that
\begin{gather}\label{eq:r_delta}
\CE_{r_\delta}\leq q_\delta.
\end{gather}
In particular,
$q_\delta+\CE_{r_\delta}\leq 2q_\delta\leq\CE_\infty$.
The quantitative Laplace principle~\cite[Proposition~3.6]{memory}, together
with Assumption~\ref{assump} (v), therefore gives
\begin{equation}
\|v_\alpha(\rho_t)-v^*\|_2
\leq
\frac{\sqrt{d}}{\eta}
\left(q_\delta+\CE_{r_\delta}\right)^\nu
+
\frac{\sqrt{d}e^{-\alpha q_\delta}}
{\rho_t(B_{r_\delta}^\infty(v^*))}
\int\|v-v^*\|_2\,d\rho_t(v).
\label{eq:laplace_mean_field}
\end{equation}}

\YMnew{First, we consider the second term in~\eqref{eq:laplace_mean_field}.}
\HZnew{To bound the denominator,
we follow the mollifier argument of~\cite[Proposition~3.8]{memory}.
By~\eqref{eq:G_mu_at_vstar}--\eqref{eq:G_mu_lip},
$G_\mu(v^*)=0$ and $G_\mu$ is globally Lipschitz. Therefore, the
additional proximal-gradient drift and diffusion can be treated by the
same argument as the gradient drift and diffusion therein. Set $
B:=c_\chi\sqrt{\CV(\rho_0)}.$ Then, on any finite time interval such that
$\|v_\alpha(\rho_t)-v^*\|_2\leq B$, there exist constants
$p_\delta>0$ and $c_\delta>0$, independent of $\alpha$, such that
\begin{gather}
\rho_t(B_{r_\delta}^\infty(v^*))
\geq c_\delta e^{-p_\delta t}.
\label{eq:mass_lower_bound}
\end{gather}
Here $c_\delta>0$ follows from
$v^*\in\operatorname{supp}(\rho_0)$.} 
\HZnew{
We now choose $\alpha>0$ sufficiently large such that
\begin{gather}\label{eq:alpha_delta}
\alpha>
\max\left\{
0,
\left[
p_\delta T^*
+
\log\left(
{2\sqrt{2d}}/({c_\chi c_\delta})
\right)
\right]/q_\delta
\right\}.
\end{gather}
Then, using Cauchy--Schwarz,
\eqref{eq:mass_lower_bound}, and~\eqref{eq:alpha_delta},
\begin{align*}
\frac{\sqrt{d}e^{-\alpha q_\delta}}
{\rho_t(B_{r_\delta}^\infty(v^*))}
\int\|v-v^*\|_2\,d\rho_t(v)
&\leq
\frac{\sqrt{2d}}{c_\delta}
e^{-\alpha q_\delta+p_\delta T^*}
\sqrt{\CV(\rho_t)}<
\frac{c_\chi}{2}\sqrt{\CV(\rho_t)}.
\end{align*}}

\HZnew{\YMnew{Next, we bound the first term in~\eqref{eq:consensus_control} by applying}
the stopping-time argument of
\cite[Section~3.4]{memory}. Before the target accuracy is reached,
$\CV(\rho_t)>\delta/2$. By~\eqref{eq:q_delta}--\eqref{eq:r_delta}, $
\frac{\sqrt{d}}{\eta}
\left(q_\delta+\CE_{r_\delta}\right)^\nu
<
\frac{c_\chi}{2}\sqrt{\CV(\rho_t)}.$ 
}

\HZnew{Thus~\eqref{eq:consensus_control} holds before the target accuracy is
reached, and
\eqref{eq:mean_field_decay} therefore yields $
\CV(\rho_t)
\leq
\CV(\rho_0)e^{-\frac{\chi}{2}t} $
until $\CV(\rho_t)\leq\delta/2$. By the definition of $T^*$, this occurs
no later than $T^*$. Hence $
\min_{t\in[0,T^*]}
\CW_2^2(\rho_t,\delta_{v^*})
\leq\delta,$
and, before the target accuracy is reached, $
\CW_2^2(\rho_t,\delta_{v^*})
\leq
\CW_2^2(\rho_0,\delta_{v^*})
e^{-\frac{\chi}{2}t}.$}
\end{proof}
\HZnew{\section{Proof of Proposition~\ref{prop:strongly-convex-main}}\label{appen: sc}}
\begin{proof}
\HZnew{We only need to refine the bound~\eqref{eq:prox_V_bound} in the proof of
Theorem~\ref{thm:mean-field long-time-behavior}.}
\HZnew{Recall in the proof of Theorem~\ref{thm:mean-field long-time-behavior},$
P_\mu(v)
=
\operatorname{prox}_{\mu g}
\left(v-\mu\nabla f(v)\right),$ and $
G_\mu(v)
=
\frac{1}{\mu}\left(v-P_\mu(v)\right).$ Since $f$ is $m_f$-strongly convex and $L_f$-smooth, by \cite[Theorem 2.1.12]{1362825894255085952}, one has
\begin{align*}
\left\langle
\nabla f(v)-\nabla f(w),v-w
\right\rangle\geq
\frac{m_fL_f}{m_f+L_f}\|v-w\|_2^2
+
\frac{1}{m_f+L_f}
\|\nabla f(v)-\nabla f(w)\|_2^2.
\end{align*}
Therefore,
\begin{align*}
&\|
(v-\mu\nabla f(v))
-
(w-\mu\nabla f(w))
\|_2^2\\
&=
\|v-w\|_2^2
-2\mu
\left\langle
\nabla f(v)-\nabla f(w),v-w
\right\rangle+
\mu^2
\|\nabla f(v)-\nabla f(w)\|_2^2\\
&\leq
\left(
1-\frac{2\mu m_fL_f}{m_f+L_f}
\right)\|v-w\|_2^2+
\left(
\mu^2-\frac{2\mu}{m_f+L_f}
\right)
\|\nabla f(v)-\nabla f(w)\|_2^2\\
&\leq
\left(
1-\frac{2\mu m_fL_f}{m_f+L_f}
\right)\|v-w\|_2^2,
\end{align*}
where the last inequality follows from
$\mu\leq 2/(m_f+L_f)$. Using the non-expansiveness of the proximal
operator, we have
\begin{gather}
\|P_\mu(v)-P_\mu(w)\|_2^2
\leq
q_\mu^2\|v-w\|_2^2,
\quad
q_\mu^2
:=
1-\frac{2\mu m_fL_f}{m_f+L_f}.
\label{eq:strong_prox_contraction}
\end{gather}}

\HZnew{Next, since $P_\mu(v^*)=v^*$ and $
\mu G_\mu(v)
=
(v-v^*)-(P_\mu(v)-v^*),$ we have
\begin{align*}
&2\mu
\left\langle v-v^*,G_\mu(v)\right\rangle=
2\left\langle
v-v^*,
(v-v^*)-(P_\mu(v)-v^*)
\right\rangle\\
&=
\|v-v^*\|_2^2
-\|P_\mu(v)-v^*\|_2^2
+\mu^2\|G_\mu(v)\|_2^2.
\end{align*}
Consequently,
\begin{align}
&-\lambda_2
\left\langle v-v^*,G_\mu(v)\right\rangle
+
\frac{\sigma_2^2}{2}\|G_\mu(v)\|_2^2
\nonumber\\
&=
-\frac{\lambda_2}{2\mu}
\left(
\|v-v^*\|_2^2
-
\|P_\mu(v)-v^*\|_2^2
\right)
-
\frac{\lambda_2\mu-\sigma_2^2}{2}
\|G_\mu(v)\|_2^2
\nonumber\\
&\leq
-\frac{\lambda_2(1-q_\mu^2)}{2\mu}
\|v-v^*\|_2^2
\nonumber\\&=
-\frac{\lambda_2m_fL_f}{m_f+L_f}
\|v-v^*\|_2^2,
\label{eq:strong_G_mu_bound}
\end{align}
where we used~\eqref{eq:strong_prox_contraction} with $w=v^*$ and
$\sigma_2^2\leq\lambda_2\mu$.}

\HZnew{Therefore, in place of~\eqref{eq:prox_V_bound}, the proximal-gradient terms
contribute $
-\frac{2\lambda_2m_fL_f}{m_f+L_f}\CV(\rho_t)$ to $\frac{d}{dt}\CV(\rho_t)$. Hence~\eqref{eq:mean_field_V_upper} becomes
\begin{align*}
\frac{d}{dt}\CV(\rho_t)
\leq{}&
-\left(
2\lambda_1-\sigma_1^2
+\frac{2\lambda_2m_fL_f}{m_f+L_f}
\right)\CV(\rho_t)
\\
&+
\sqrt{2}(\lambda_1+\sigma_1^2)
\sqrt{\CV(\rho_t)}
\|v_\alpha(\rho_t)-v^*\|_2+
\frac{\sigma_1^2}{2}
\|v_\alpha(\rho_t)-v^*\|_2^2.
\end{align*}
The remaining quantitative Laplace principle, mass lower-bound, and
stopping-time arguments in the proof of
Theorem~\ref{thm:mean-field long-time-behavior} are unchanged.
Combining the resulting mean-field estimate with
Theorem~\ref{thm:quantitative-mean-field-limit}, exactly as in the proof of
Theorem~\ref{main}, gives the result.}
\end{proof}
\section{Proof of Theorem~\ref{thm:wellposed for micro}}\label{app:proof of micro}
Finally, we present the proof of Theorem~\ref{thm:wellposed for micro}, which follows standard techniques.
\begin{proof}[Proof of Theorem~\ref{thm:wellposed for micro}]
The proof largely follows \cite[Theorem 2.1]{cbo2} and \cite[Theorem 2.2]{gerber2023mean}. We use $\widehat{\rho}^N_t$ to denote the empirical measure of $V^{1,N}_t,\dots,V^{N,N}_t$. Since this is an existence and uniqueness result, WLOG, we assume $\lambda_1=\lambda_2=\sigma_1=\sigma_2=1$. We can concatenate $ \left\{V_{t}^{i,N}\right\}_{i=1}^{N} $ into one vector and put them in one equation. To be specific, we define
\begin{gather*}
	V_{t}=\left(\left(V_{t}^{1,N}\right)^{\top},...,\left(V_{t}^{N,N}\right)^{\top}\right)^{\top}.
\end{gather*}
Then $ V_{t} $ is a vector in $ \mathbb{R}^{Nd} $ for each fixed $ t $ and it will satisfy the following equation:
\begin{gather}\label{oneeq}
	dV_{t}=\left(F_{N}(V_{t})+\widetilde{F}_N(V_{t})\right)\,dt+M_{N}(V_{t})\,dB_{t}^{(N)}.
\end{gather}
Here $ B^{(N)} $ is the standard Wiener process in $ \mathbb{R}^{2Nd} $. \begin{gather*}
	F_{N}(V_t)=\left(\left(F_{N}^{1}\left(V_t\right)\right)^{\top},...,\left(F_{N}^{N}\left(V_t\right)\right)^{\top}\right)^{\top}\in\mathbb{R}^{Nd},
\end{gather*}
where $
	F_{N}^{i}(V_t)=-\left(V^{i,N}_t - v_\alpha(\widehat{\rho}^N_t)\right)\in\mathbb{R}^{d}.$
 Further,
 \begin{gather*}
	\widetilde{F}_{N}(V_t)=\left(\left(\widetilde{F}_{N}^{1}\left(V_t\right)\right)^{\top},...,\left(\widetilde{F}_{N}^{N}\left(V_t\right)\right)^{\top}\right)^{\top}\in\mathbb{R}^{Nd},
\end{gather*}
where
\begin{gather*}
	\widetilde{F}_{N}^{i}(V_t)=-\left(\nabla f(V^{i,N}_t) + \nabla M_{\mu g} \left(V_t^{i,N} - \mu \nabla f(V^{i,N}_t)\right)\right).
 \end{gather*}
 Moreover,
$
M_{N}(V_t)=\left(\text{diag}\left(F_N(V_t)\right),\text{diag}\left(\widetilde{F}_N(V_t)\right)\right)\in\mathbb{R}^{Nd\times 2Nd}\HZnew{.}
$
Thus it suffices to prove the well-posedness result of equation \eqref{oneeq}. From \cite[Lemma 2.1]{cbo2}, $F_N$ is locally Lipschitz and satisfies $\|F_N(V)\|_2\leq C\|V\|_2$. Also, \HZnew{by Assumption~\ref{assump} (iii) and the $1/\mu$-Lipschitz continuity of $\nabla M_{\mu g}$, $\widetilde{F}_N$ is globally Lipschitz, and consequently satisfies
$\|\widetilde{F}_N(V)\|_2\leq C(1+\|V\|_2)$.} By \cite[Theorem 3.5]{khasminskii2012stochastic}, it suffices to find a function $\phi\in \mathcal{C}^2(\R^{Nd},[0,\infty))$ such that
\begin{itemize}
\item $\lim_{\|v\|_2\rightarrow \infty}\phi(v) = +\infty.$
\item There is a constant $c$ such that for all $V\in \R^{Nd}$,
\begin{align*}
\CL \phi (V)
&:= \left(F_{N}(V) + \widetilde{F}_N (V)\right)\cdot \nabla \phi(V)+ \frac{1}{2}\text{tr}\left(M_N(V)^\top \nabla^2 \phi(V) M_N(V)\right)\leq c \phi(V).
\end{align*}
\end{itemize}
We pick $\phi(V) = \tfrac{1}{2}\|V\|_2^2 + 1$. The first is trivially satisfied. For the second, using the facts that $\|F_N(V)\|_2\leq C \|V\|_2$ and $\|\widetilde{F}_N (V)\|_2\leq C (1+ \|V\|_2)$, one can verify that $
\CL \phi (V) \lesssim (1 + \|V\|_2)\|V\|_2 + \frac{1}{2} (\|V\|_2^2 +  (1+\|V\|_2)^2) 
\leq c \phi(V).$ 
\end{proof}
\clearpage
\section*{Supplementary Materials}
This section contains all lengthy and technical proofs
for the results presented in the main paper.

\renewcommand{\sectionname}{Supplementary}
\setcounter{section}{0}
\renewcommand{\thesection}{S\arabic{section}}
\renewcommand{\theHsection}{supp.\arabic{section}}
\renewcommand{\theHtheorem}
  {supp.\arabic{section}.\arabic{theorem}}

\setcounter{algorithm}{0}
\renewcommand{\thealgorithm}{S.\arabic{algorithm}}
\providecommand{\theHalgorithm}{\thealgorithm}
\renewcommand{\theHalgorithm}{supp.\arabic{algorithm}}
\section{Technical Proofs in Appendix~\ref{app:stab}}\label{B}
\begin{proof}[Proof of Lemma~\ref{hglemma}]
From the proof of \cite[Lemma A.1]{gerber2023mean}, 
\begin{align*}
&\left\| \frac{\int g \,d\mu}{\int h \,d\mu} - \frac{\int g \,d\nu}{\int h \,d\nu}\right\|
\leq L\left(\eta^{-1} + \eta^{-2}\int \|g\|\,d\nu\right)\\
&\times\left(\int\int (1+\|u\|+\|v\|)^{\frac{\xi p}{p-1}}\,d\mu(u)\,d\nu(v)\right)^{\frac{p-1}{p}} \CW_p(\mu,\nu).
\end{align*}
Notice 
\begin{align*}
\int \|g\|\,d\nu
&\leq \int \left(\|g(0)\| + L(1+\|u\|)^\xi\right)\|u\| \,d\nu\\ 
&\leq \|g(0)\|\int \|u\|\, d\nu + L \int\left(1+\|u\|\right)^{\xi+1}\,d\nu\\
&\leq \|g(0)\| \left(\int \|u\|^p\, d\nu\right)^{\frac{1}{p}} \!\!+ L \int\left(1+\|u\|\right)^p\,d\nu\\
&\leq  \|g(0)\| \left(\int \|u\|^p\, d\nu\right)^{\frac{1}{p}} \!\!+ 2^{p-1}L\left(\! 1 +\! \int \|u\|^p \, d\nu \right)\\
&\leq  C_{p,L} \left(\|g(0)\|R + R^p + 1\right).
\end{align*}
Moreover, since $\frac{\xi p}{p-1}\leq p$ by the condition $p\geq \xi+1$, we have
\begin{align*}
&\left(\int\int (1+\|u\|+\|v\|)^{\frac{\xi p}{p-1}}\,d\mu(u)\,d\nu(v)\right)^{\frac{p-1}{p}}\\
&\qquad\leq \left(\int\int (1+\|u\|+\|v\|)^{p}\,d\mu(u)\,d\nu(v)\right)^{\frac{p-1}{p}}\\
&\qquad\leq C_p (1+R^p)^{\frac{p-1}{p}}\leq C_p(1+R^{p-1}).
\end{align*}
The above inequalities finish the proof.
\end{proof}

\section{Technical Proofs in Appendix~\ref{app:moment}}\label{C}
\begin{proof}[Proof of Proposition~\ref{prop:micro_moment_bound}]
{We will bound each of the three terms in \eqref{three_moment_bounds}. To start with, consider the first term.}
We first note by Proposition~\ref{A4}, it holds that
\begin{align*}
 &\left\|-\lambda_1 \left(V^{i,N}_t - v_\alpha (\widehat{\rho}^N_t)\right)\right\|_2\\
 &\leq \lambda_1 (\|V^{i,N}_t\|_2 + \|v_\alpha(\widehat{\rho}^N_t)\|_2)\\
 &\leq C \left(\|V^{i,N}_t\|_2 + \left(\int \|v\|_2^p\,d\widehat{\rho}^N_t\right)^{1/p}\right).
\end{align*}
Moreover, by the relation between $\nabla M_{\mu g}$ and $\mathrm{prox}_{\mu g}$ in \eqref{eq:moreau_prox}, we can verify that $\nabla f(v^*) + \nabla M_{\mu g} \left(v^* - \mu \nabla f(v^*)\right)=0$. Together with the Lipschitz properties of $\nabla f$ and $\mathrm{prox}_{\mu g}$, 
\begin{align*}
&\left\|-\lambda_2 \left( \nabla f(V^{i,N}_t) + \nabla M_{\mu g} \left(V_t^i - \mu \nabla f(V^{i,N}_t)\right)\right)\right\|_2\\
&\leq \lambda_2 \Bigg(L_f \|V^{i,N}_t - v^*\|_2+ \frac{1}{\mu}\left(\|V^{i,N}_t-v^*\|_2 + \mu L_f \|V^{i,N}_t - v^*\|_2\right)\Bigg)\\
&\leq C (\|v^*\|_2+\|V^{i,N}_t\|_2).
\end{align*}
Similarly, one can deduce
\begin{align*}
 \left\|\sigma_1 D\left(V^{i,N}_t - v_\alpha (\widehat{\rho}^N_t)\right)\right\|_F\leq C \left(\|V^{i,N}_t\|_2 + \left(\int \|v\|_2^p\,d\widehat{\rho}^N_t\right)^{1/p}\right).
\end{align*}
and
\begin{align*}
\left\|\sigma_2 \left( \nabla f(V^{i,N}_t) + \nabla M_{\mu g} \left(V_t^{i,N} - \mu \nabla f(V^{i,N}_t)\right)\right)\right\|_2\leq C (\|v^*\|_2+\|V^{i,N}_t\|_2).
\end{align*}
Then {following similar steps in \eqref{holder_pre} with $u_s = v_\alpha(\widehat{\rho}^N_s)$, 
\begin{align*}
\E \left[\sup_{s\in[r,t]} \|V^{i,N}_s - V^{i,N}_r\|_2^p\right]\leq C \left[(t-r)^{p-1} + (t-r)^{\tfrac{p}{2} - 1}\right] \times\int_r^t \left(\|v^*\|_2^p + \|v_\alpha(\widehat{\rho}^N_s)\|_2^p + \E\|V^{i,N}_s\|_2^p\right) \,ds.
\end{align*}
Then setting $r=0$ and following the same step as the first inequality in \eqref{holder_mid}}, one can deduce
\begin{equation*}
\begin{aligned}
&\E \left[\sup_{s\in[0,t]} \|V^{i,N}_s\|_2^p\right] \\
 &\leq C \cdot\Bigg( \E \|V^{i,N}_0\|_2^p+\Big(T^{p-1} + T^{\frac{p}{2} - 1}\Big)\times\int_0^t \Big(\|v^*\|_2^p + \|v_\alpha(\widehat{\rho}^N_s)\|_2^p +  \E \|V^{i,N}_s\|_2^p\Big) \,ds\Bigg)\\
 &\overset{(a)}{\leq} C \cdot\Bigg( \E \|V^{i,N}_0\|_2^p + \Big(T^{p-1} + T^{\frac{p}{2} - 1}\Big)\times\int_0^t \Big(\|v^*\|_2^p + \int \|v\|_2^p\,d\widehat{\rho}^N_s +  \E \|V^{i,N}_s\|_2^p\Big) \,ds\Bigg),
 \end{aligned}
 \end{equation*}
 where (a) is due to Proposition~\ref{A4}.
Then one has
\begin{align*}
&\E \left[\sup_{s\in[0,t]} \|V_s^{i,N}\|_2^p\right]\leq C \Bigg(\E [\|V^{i,N}_0\|_2^p]+k_p(T)\cdot \|v^*\|_2^p+ k_p(T)\cdot  \int_0^t \E\left[\|V^{i,N}_s\|_2^p + \int \|v\|_2^p\,d\widehat{\rho}^N_s\right] \,ds\Bigg).
\end{align*}
Also, by the permutation-invariance of the empirical measure, 
\begin{gather*}
\E \int \|v\|_2^p \,d\widehat{\rho}^N_s = \E \|V^{i,N}_s\|_2^p.
\end{gather*}
Thus,
\begin{align*}
\E \left[\sup_{s\in[0,t]} \|V_s^{i,N}\|_2^p\right]&\leq C \Bigg(\E [\|V^{i,N}_0\|_2^p]+k_p(T)\cdot \|v^*\|_2^p+ k_p(T)\cdot  \int_0^t \E\left[\|V^{i,N}_s\|_2^p \right] \,ds\Bigg)\\
&\leq C \Bigg(\E [\|V^{i,N}_0\|_2^p]+k_p(T)\cdot \|v^*\|_2^p+ k_p(T)\cdot  \int_0^t \E\left[\sup_{r\in[0,s]}\|V^{i,N}_r\|_2^p \right] \,ds\Bigg).
\end{align*}
Gr\"onwall's inequality then yields
\begin{align*}
&\E \left[\sup_{s\in[0,T]} \|V_s^{i,N}\|_2^p\right]\leq C \left(\E [\|V^{i,N}_0\|_2^p]+k_p(T)\cdot \|v^*\|_2^p\right)\cdot e^{C \cdot T\cdot k_p(T)}.
\end{align*} 
The second term in \eqref{three_moment_bounds} can be bounded as follows
\begin{align*}
\E \left[\sup_{t\in[0,T]}\int \|v\|_2^p\,d\widehat{\rho}^N_t\right]&=\E \left[\sup_{t\in[0,T]}\frac{1}{N} \sum_{i=1}^N\|V^{i,N}_t\|_2^p\right]\\
&\leq \E \left[\frac{1}{N} \sum_{i=1}^N \sup_{t\in[0,T]}\|V^{i,N}_t\|_2^p\right]\\
&=\E \left[\sup_{t\in[0,T]} \|V_t^{i,N}\|_2^p\right].
\end{align*}
And the {bound for the third term in \eqref{three_moment_bounds}} follows from Proposition~\ref{A4}.
\end{proof}
\section{Technical Proofs in Appendix~\ref{app:proof of macro}}\label{D}
\begin{proof}[Proof of Lemma~\ref{lemma:well-definedness of T}]
We omit the existence and uniqueness as one can easily obtain those results via \cite[Theorem 5.2.1]{oksendal2003stochastic} thanks to the boundedness of $u$. We only prove the continuity of $v_\alpha (\rho_t)$ and the expectation bound here. Fix $0\leq r\leq t\leq T$, 
\begin{equation}\label{holder_pre}
\begin{aligned}
&\E \left[\sup_{s\in[r,t]} \|\overline{V}_s - \overline{V}_r\|_2^p\right]\\
&\overset{(a)}{\leq} 2^{p-1} \E \Bigg[\Bigg(\int_r^t \Big\|\lambda_1 \left(\overline{V}_s - u_s\right) + \lambda_2 \nabla f(\overline{V}_s)+ \lambda_2\nabla M_{\mu g} \left(\overline{V}_s - \mu \nabla f(\overline{V}_s)\right)\Big\|_2 \,ds\Bigg)^p \Bigg]\\
&\quad\quad\quad\quad\quad+C_{\rm BDG}\cdot2^{p-1} \E \Bigg[\Bigg(\int_r^t \Big\|\sigma_1 D(\overline{V}_s - u_s)\Big\|_F^2\, ds\\ 
&\qquad\qquad\qquad\qquad\qquad\qquad+ \int_r^t \Big\|\sigma_2D\left( \nabla f(\overline{V}_s) + \nabla M_{\mu g} \left(\overline{V}_s - \mu \nabla f(\overline{V}_s)\right)\right)\Big\|_F^2 \,ds\Bigg)^{p/2}\Bigg]\\
&\overset{(b)}{\leq} C \left[(t-r)^{p-1} + (t-r)^{\tfrac{p}{2} - 1}\right]\times \E \left[\int_r^t \left(\|v^*\|_2^p + \|\overline{V}_s - u_s\|_2^p + \|\overline{V}_s\|_2^p\right) \,ds\right]\\
&\leq C \left[(t-r)^{p-1} + (t-r)^{\tfrac{p}{2} - 1}\right]\times\int_r^t \left(\|v^*\|_2^p + \|u_s\|_2^p + \E\|\overline{V}_s\|_2^p\right) \,ds,
\end{aligned}
\end{equation}
where step $(a)$ follows from Burkholder-Davis-Gundy inequality~\cite[Theorem 7.3]{mao2007stochastic} and the constant $C_{\rm BDG}$ only depends on $p$, and step $(b)$ follows from global Lipschitzness of $\nabla f$, $\nabla M_{\mu g}$ and $D$, $\nabla f(v^*) + \nabla M_{\mu g} \left(v^* - \mu \nabla f(v^*)\right) = 0$, and H\"older's inequality. Now we take $r=0$ to obtain
\begin{equation}\label{holder_mid}
\begin{aligned}
&\E \left[\sup_{s\in[0,t]} \|\overline{V}_s\|_2^p\right] \\
 &\leq C \cdot \Bigg(\E \|\overline{V}_0\|_2^p+ \Big(T^{p-1} + T^{\frac{p}{2} - 1}\Big)\int_0^t \Big(\|v^*\|_2^p + \|u_s\|_2^p +  \E \|\overline{V}_s\|_2^p\Big) \,ds\Bigg)\\
&\leq C \cdot \Bigg(\E \|\overline{V}_0\|_2^p  + \Big(T^p + T^{\frac{p}{2}}\Big)\cdot \Big(\|v^*\|_2^p+ \|u\|^p_{L^\infty([0,T])}\Big)+ \left(T^{p-1}+T^{\frac{p}{2} - 1}\right)\cdot \int_0^t \E\left[ \sup_{r\in[0,s]}\|\overline{V}_r\|_2^p\right]\,ds\Bigg)\\
& \leq C \cdot \Bigg(\E \|\overline{V}_0\|_2^p  +  k_p(T)\Big(\|v^*\|_2^p+ \|u\|^p_{L^\infty([0,T])}\Big)+  k_p(T)\int_0^t \E\left[ \sup_{r\in[0,s]}\|\overline{V}_r\|_2^p\right]\,ds\Bigg).
\end{aligned}
\end{equation}
Gr\"onwall's inequality then gives 
\begin{align*}
&\E \left[\sup_{s\in[0,T]} \|\overline{V}_s\|_2^p\right]\leq C\left(\E \|\overline{V}_0\|_2^p  +  k_p(T)\left(\|v^*\|_2^p+ \|u\|^p_{L^\infty([0,T])}\right)\right) e^{C\cdot T\cdot k_p(T)}.
\end{align*}
This gives the expectation bound. {Note that $\|\overline{V}_s\|_2\leq\|\overline{V}_s\|_2^p$ if $\|\overline{V}_s\|_2\geq 1$.}
Thus for any $s\in[0,T]$, one has
\begin{align*}
\|\overline{V}_s e^{-\alpha \CE(\overline{V}_s)}\|_2 &\leq e^{-\alpha \underline{\CE}}\|\overline{V}_s\|_2\leq e^{-\alpha \underline{\CE}} \max\left\{1,\sup_{s\in[0,T]}\|V_s\|_2^p\right\} \in L^1 (\Omega)
\end{align*}
and $e^{-\alpha \CE(\overline{V}_s)} \leq e^{-\alpha \underline{\CE}}\in L^1(\Omega)$.
By dominated convergence theorem, for $s\in[0,T]$,
\begin{gather*}
\lim_{r\rightarrow s}v_\alpha (\rho_r) = \lim_{r\rightarrow s}\frac{\E \overline{V}_r e^{-\alpha \CE(\overline{V}_r)}}{\E e^{-\alpha \CE( \overline{V}_r)}} = \frac{\E \overline{V}_s e^{-\alpha \CE(\overline{V}_s)}}{\E e^{-\alpha \CE(\overline{V}_s)}} = v_\alpha (\rho_s).
\end{gather*}
This proves the continuity.
\end{proof}
\section{Technical Proofs in Appendix~\ref{app:proof of quantitative}}\label{E}
\begin{proof}[Proof of Lemma~\ref{lem:first_term}] Let
\begin{equation*}
K:= C\cdot \left(\E \left[ \|V_0\|_2^2 \right]+k_2(T)\|v^*\|_2^2\right)\cdot e^{C\cdot T\cdot k_2(T)}.
\end{equation*}
From Theorem~\ref{thm:mean-fild existence theorem}, we know that for all $i$,
\begin{gather*}
\E\left[\sup_{t\in[0,s]}\left\|\overline{V}^{i,N}_{t}\right\|_2^2\right] \leq K.
\end{gather*}
Thus
\begin{gather*}
\E\left[\frac{1}{N}\sum_{i=1}^N \sup_{t\in[0,s]} \left\|\overline{V}_t^{i,N}\right\|_2^2\right]\leq K.
\end{gather*}Then we consider the set
\begin{gather*}
\Omega^N_{s}=\left\{\frac{1}{N}\sum_{i=1}^N \sup_{t\in[0,s]} \left\|\overline{V}_t^{i,N}\right\|_2^2 \geq K + 1\right\}.
\end{gather*}
From \cite[Lemma 2.5]{gerber2023mean} with $Z_{i} = \displaystyle\sup_{t\in[0,s]}\left\|\overline{V}^{i,N}_{t}\right\|_2^2$, $R=K+1$ and $r=4$, one has
\begin{equation}\label{probability_bound}
\begin{split}
\P (\Omega^{N}_s)
&\leq C  \E\left[|Z_1 - \E [Z_1]|^4\right]  N^{-2}\\
&\leq C  \E\left[|Z_1|^4\right]  N^{-2}\\ 
&= C \E\left[\sup_{t\in[0,s]} \left\|\overline{V}^{1,N}_t\right\|_2^8\right] N^{-2}\\
&\leq C\!\left(\E \left[ \|V_0\|_2^8 \right]\!+\!k_8(T)\|v^*\|_2^8\right) e^{C\cdot T\cdot k_8(T)} N^{-2},
\end{split}
\end{equation}
where $C$ is an absolute constant and we used Theorem~\ref{thm:mean-fild existence theorem} in the last inequality. We can then compute 
\begin{equation}\label{omega split}
\begin{split}
&\E\left[\left\|v_\alpha (\widehat{\rho}^N_s) - v_\alpha (\widehat{\overline{\rho}}^N_s)\right\|_2^2\right]\\
&= \E\left[\left\|v_\alpha (\widehat{\rho}^N_s) - v_\alpha (\widehat{\overline{\rho}}^N_s)\right\|_2^2 \mathbbm{1}_{\Omega^N_s}\right]+ \E\left[\left\|v_\alpha (\widehat{\rho}^N_s) - v_\alpha (\widehat{\overline{\rho}}^N_s)\right\|_2^2\mathbbm{1}_{(\Omega^{N}_s)^c}\right].
\end{split}
\end{equation}
The motivation for this splitting is that the event $\Omega_s^N$ occurs with probability decaying in $N$, so the first term yields the desired dependence on the number of particles. Although the complement event $(\Omega_s^N)^c$ may have large probability, conditional on this event we can bound $
\E\left[\big\|v_\alpha(\widehat{\rho}^N_s) - v_\alpha(\widehat{\overline{\rho}}^N_s)\big\|_2^2 \,\mathbbm{1}_{(\Omega^N_s)^c}\right]$ by a constant multiple of the Wasserstein distance between $\widehat{\overline{\rho}}^N_s$ and $\widehat{\rho}^N_s$, by virtue of Proposition~\ref{prop:stab}. This quantity, in turn, can be bounded by $\E[\|V^{1,N}_s - \overline{V}^{1,N}_s\|_2^2]$, which enables the subsequent application of Grönwall’s inequality.

\paragraph{\textbf{Upper bound of $\E\left[\left\|v_\alpha (\widehat{\rho}^N_s) - v_\alpha (\widehat{\overline{\rho}}^N_s)\right\|_2^2 \mathbbm{1}_{\Omega^N_s}\right] $}} By H\"older inequality, one has
\begin{align*}
&\E\left[\left\|v_\alpha (\widehat{\rho}^N_s) - v_\alpha (\widehat{\overline{\rho}}^N_s)\right\|_2^2 \mathbbm{1}_{\Omega^N_s}\right]\\
&\leq  \E\left[\left\|v_\alpha (\widehat{\rho}^N_s) - v_\alpha (\widehat{\overline{\rho}}^N_s)\right\|_2^8 \right]^{1/4}\cdot \P(\Omega^N_s)^{3/4}\\
&\leq \E\left[\left\|v_\alpha (\widehat{\rho}^N_s) - v_\alpha (\widehat{\overline{\rho}}^N_s)\right\|_2^8 \right]^{1/4}\cdot \P(\Omega^N_s)^{1/2}.
\end{align*}
Then by Proposition~\ref{A4}, Proposition~\ref{prop:micro_moment_bound} and Theorem~\ref{thm:mean-fild existence theorem},
\begin{align*}
&\E\left[\left\|v_\alpha (\widehat{\rho}^N_s) - v_\alpha (\widehat{\overline{\rho}}^N_s)\right\|_2^8 \right]\\
&\qquad\leq C \cdot \left(\E\left[\left\|v_\alpha (\widehat{\rho}^N_s)\right\|_2^8\right] + \E\left[\left\|v_\alpha (\widehat{\overline{\rho}}^N_s)\right\|_2^8\right]\right)\\
&\qquad\leq C\cdot \E\left[\int \|v\|_2^8 \,d\widehat{\rho}_s^N +\int \|v\|_2^8 \,d\widehat{\overline{\rho}}^N_s \right]\\
&\qquad=C\cdot \left(\E\left[\left\|V^{1,N}_s\right\|_2^8\right] +\E\left[\left\|\overline{V}^{1,N}_s\right\|_2^8\right] \right)\\
&\qquad\leq C\cdot \left(\E \left[ \|V_0\|_2^8 \right]+k_8(T)\|v^*\|_2^8\right)\cdot e^{C\cdot T\cdot k_8(T)}.
\end{align*}
Thus combining the above inequality and \eqref{probability_bound},
\begin{equation}\label{omega}
\begin{split}
&\E\left[\left\|v_\alpha (\widehat{\rho}^N_s) - v_\alpha (\widehat{\overline{\rho}}^N_s)\right\|_2^2 \mathbbm{1}_{\Omega^N_s}\right]\\
&\leq C\left(\E \left[ \|V_0\|_2^8 \right]+k_8(T)\|v^*\|_2^8\right)^{3/4} e^{C\cdot T\cdot k_8(T)} N^{-1}.
\end{split}
\end{equation}
\paragraph{\textbf{Upper bound of $\E\left[\left\|v_\alpha (\widehat{\rho}^N_s) - v_\alpha (\widehat{\overline{\rho}}^N_s)\right\|_2^2 \mathbbm{1}_{(\Omega^N_s)^c}\right] $}} By the definition of the set $(\Omega^N_s)^c$, for all paths sampled from $(\Omega^N_s)^c$, we know $\widehat{\overline{\rho}}^N_s \in \CP_{2, \sqrt{K+1}}(\R^d)$. Also, from Proposition~\ref{prop:micro_moment_bound}, $\widehat{\rho}^N_s\in\CP_{2}(\R^d)$. Then by Proposition~\ref{prop:stab}, 
\begin{equation}\label{omega_c}
\begin{aligned}
&\E\left[\left\|v_\alpha (\widehat{\rho}^N_s) - v_\alpha (\widehat{\overline{\rho}}^N_s)\right\|_2^2 \mathbbm{1}_{(\Omega^N_s)^c}\right]\\
&\leq C \left(1 + e^{C(1+(2\sqrt{K+1})^{l} )}\left(1+(K+1)^{\tfrac{3}{2}}\right)\right)^2\times\E\left[ \CW_2(\widehat{\rho}^N_s,\widehat{\overline{\rho}}^N_s)\right ]^2\\
&\leq C \left(1 + e^{C(1+(2\sqrt{K+1})^{l} )}\left(1+(K+1)^{\tfrac{3}{2}}\right)\right)^2\times\E\left[\left\|V^{1,N}_s - \overline{V}^{1,N}_s\right\|_2^2\right].
\end{aligned}
\end{equation}
Combining \eqref{omega split}, \eqref{omega} and \eqref{omega_c}, one can finish the proof.
\end{proof}
\begin{proof}[Proof of Theorem~\ref{thm:MSEbound}]
The proof largely follows \cite[Theorem 2.3]{agapiou2017importance}. By \cite[Lemma 6.4]{agapiou2017importance} (here we used a stronger version of it, namely, the last line in the proof of \cite[Lemma 2]{doukhan2009evaluation}), one has
\begin{gather*}
\E\left[\left\|\frac{\int \phi g \,d\widehat{\mu}^N}{\int g \,d\widehat{\mu}^N} - \frac{\int \phi g \,d\mu}{\int g \,d\mu}\right\|_2^2\right] \leq 3\sum_{k=1}^d (A_{1,k}+A_{2,k}+A_{3,k}),
\end{gather*}
where
\begin{align*}
A_{1,k} &= \frac{1}{(\int g\,d\mu)^2}\E\left[\left(\int \phi_k g \,d\widehat{\mu}^N - \int \phi_k g\,d\mu\right)^2\right],\\
A_{2,k} &=\frac{1}{(\int g \,d\mu)^4}\E\left[\left|\int \phi_k g\,d\widehat{\mu}^N\left(\int g\,d\widehat{\mu}^N - \int g\,d\mu\right)\right|^2\right],\\
A_{3,k} &=  \frac{1}{(\int g\,d\mu)^2(1+\theta)}\E\Bigg[ \left|\frac{\sum_{i=1}^N |\phi_k(u^{i,N})| g(u^{i,N})}{\sum_{i=1}^N g(u^{i,N})}\right|^2\times \left|\int g\,d\mu - \int g\,d\widehat{\mu}^N\right|^{2(1+\theta)}\Bigg].
\end{align*}
In the above, $\theta\in (0,1)$ and its choice will be specified later.
\paragraph{\textbf{Upper bound of $\sum_{k=1}^d A_{1,k}$}} One has
\begin{equation}
\begin{aligned}\label{A1k}
\sum_{k=1}^d A_{1,k}&= \frac{1}{(\int g\,d\mu)^2}\sum_{k=1}^d \E\left[\left(\int \phi_k g \,d\widehat{\mu}^N - \int \phi_k g\,d\mu\right)^2\right]\\
&=\frac{1}{(\int g\,d\mu)^2}\sum_{k=1}^d\E\left[\left(\int \left(\phi_k g-\int \phi_k g\,d\mu\right) \,d\widehat{\mu}^N\right)^2\right]\\
&=\frac{N^{-2}}{(\int g\,d\mu)^2}\sum_{k=1}^d\sum_{i=1}^N\E\left[\left( \phi_k(u^i) g(u^i)-\int \phi_k g\,d\mu \right)^2\right]\\
&\leq N^{-1}\cdot \frac{1}{(\int g\,d\mu)^2}\sum_{k=1}^d \CM_{2}(\phi_k g).\\
\end{aligned}
\end{equation}
\paragraph{\textbf{Upper bound of $\sum_{k=1}^d A_{2,k}$}} 
Let 
\begin{equation*}
h = \int g\,d\widehat{\mu}^N - \int g\,d\mu.
\end{equation*}
Then one has
\begin{align*}
\sum_{k=1}^d A_{2,k}
&= \frac{1}{(\int g \,d\mu)^4}\E\left[\sum_{k=1}^d\left|\int \phi_k g\,d\widehat{\mu}^N\cdot h\right|^2\right]\\
&= \frac{1}{(\int g \,d\mu)^4} \E\left[h^2\cdot \left(\sum_{k=1}^d \left|\int \phi_k g\,d\widehat{\mu}^N\right|^2\right)\right]\\
&\overset{(a)}{\leq} \frac{1}{(\int g \,d\mu)^4} \E\left[h^2\cdot \left(\int \left(\sum_{k=1}^d g^2\left|\phi_k\right|^2\right)^{1/2}\,d\widehat{\mu}^N\right)^2\right]\\
&=\frac{1}{(\int g \,d\mu)^4} \E\left[h^2\cdot \left(\int  g \|\phi\|_2\,d\widehat{\mu}^N\right)^2\right]
\end{align*}
\begin{align*}
&\overset{(b)}{\leq} \frac{1}{(\int g \,d\mu)^4} \E\left[h^{2m}\right]^{1/m} \cdot \E\left[\left(\int  g \|\phi\|_2\,d\widehat{\mu}^N\right)^{2l}\right]^{1/l},
\end{align*}
where step $(a)$ follows from Minkowski inequality applied to integral against $\widehat{\mu}^N$ and integral against the counting measure of $k$, and step $(b)$ follows from H\"older inequality. Moreover,
\begin{align*}
&\E\left[\left(\int  g \|\phi\|_2\,d\widehat{\mu}^N\right)^{2l}\right]^{1/l}\\
&= \frac{1}{N^2} \E\left[\left(\sum_{i=1}^N g(u^{i,N})\|\phi(u^{i,N})\|_2\right)^{2l}\right]^{1/l}\\
&\leq \frac{1}{N^2} \left(\sum_{i=1}^N \E\left[g(u^{i,N})^{2l}\|\phi(u^{i,N})\|_2^{2l}\right]^{1/2l}\right)^{2}\\
&=\left(\int \left(g\|\phi\|_2\right)^{2l}\,d\mu\right)^{1/l},
\end{align*}
where we used Minkowski inequality applied to integral against $\mu$ (denoted by $\E[\cdot]$) and integral against the counting measure of $i$. By \cite[Equation (6.2)]{agapiou2017importance}, one has
\begin{align*}
&\E\left[\left(\int g\,d\widehat{\mu}^N - \int g\,d\mu\right)^{2m}\right]^{1/m}\\
&\leq C_{2m}^{1/m} \left(\int \left(g - \int g\,d\mu\right)^{2m}\,d\mu\right)^{1/m} \cdot N^{-1}\\
&=C^{1/m}_{2m} \CM_{2m}^{1/m}(g) \cdot N^{-1}.
\end{align*}
Thus
\begin{equation}\label{A2K}
\begin{split}
\sum_{k=1}^d A_{2,k}\leq \frac{1}{(\int g \,d\mu)^4}\left(\int \left(g\|\phi\|_2\right)^{2l}\,d\mu\right)^{1/l} C^{1/m}_{2m} \CM_{2m}^{1/m}(g) \cdot N^{-1} .
\end{split}
\end{equation}
\paragraph{\textbf{Upper bound of $\sum_{k=1}^d A_{3,k}$}} 
Let
\begin{equation*}
\gamma_{k} = \frac{\sum_{i=1}^N |\phi_k(u^{i,N})| g(u^{i,N})}{\sum_{i=1}^N g(u^{i,N})}.
\end{equation*}
Then one has
\begin{align*}
\sum_{k=1}^d A_{3,k}&=\sum_{k=1}^d \frac{\E\left[\left|\gamma_{k}\right|^2 \cdot \left|\int g\,d\mu - \int g\,d\widehat{\mu}^N\right|^{2(1+\theta)}\right]}{(\int g\,d\mu)^2(1+\theta)}\\
&= \frac{\E\left[\left(\sum_{k=1}^d\left|\gamma_{k}\right|^2\right) \cdot \left|\int g\,d\mu - \int g\,d\widehat{\mu}^N\right|^{2(1+\theta)}\right]}{(\int g\,d\mu)^2(1+\theta)}.
\end{align*}
Use $w^{i,N}$ to denote $\tfrac{g(u^{i,N})}{\sum_{i=1}^N g(u^{i,N})}$, and $w^{N}$ to denote the vector $(w^{1,N},\dots,w^{N,N})^\top \in \R^{N}$. One knows $0 < w^{i,N} <1$ and $\sum_{i=1}^N w^{i,N} = 1$. Further, we use $\Phi^N$ to denote the matrix $(|\phi(u^{1,N})|,\dots,|\phi(u^{N,N})|)\in \R^{d\times N}$. Here when the absolute value symbol $|\cdot|$ is applied to a vector, it means entry-wise application. We have
\begin{align*}
\sum_{k=1}^d\left|\gamma_{k}\right|^2
&=\sum_{k=1}^d\left|\frac{\sum_{i=1}^N |\phi_k(u^{i,N})| g(u^{i,N})}{\sum_{i=1}^N g(u^{i,N})}\right|^2\\
&=\sum_{k=1}^d\sum_{r,s=1}^{N} |\phi_{k}(u^{r,N})|\cdot w^{r,N} \cdot |\phi_{k}(u^{s,N})|\cdot w^{s,N}\\
&=\sum_{r,s=1}^N w^{r,N}\cdot w^{s,N} \cdot \left( |\phi(u^{r,N})|^\top \cdot |\phi(u^{s,N})|\right)\\
&=\sum_{r,s=1}^N w^{r,N}\cdot w^{s,N} \cdot \left((\Phi^N)^T \Phi^N\right)_{r,s}\\
&= \left\|\Phi^N w^N\right\|_2^2.
\end{align*}
Also, we know
\begin{align*}
\|\Phi^N w^N\|_2&\leq w^{1,N}\|\phi(u^{1,N})\|_2+\dots+w^{N,N}\|\phi(u^{i,N})\|_2\leq \max_{1\leq i\leq N} \|\phi(u^{i,N})\|_2.
\end{align*}
Thus
\begin{align*}
\sum_{k=1}^d A_{3,k}
&\leq\! \frac{\E\!\left[\displaystyle\max_{1\leq i\leq N} \|\phi(u^{i,N})\|_2^2 \cdot \left|\int g\,d\mu - \int g\,d\widehat{\mu}^N\right|^{2(1+\theta)}\right]}{(\int g\,d\mu)^2(1+\theta)}\\
&\leq\!  \frac{\E\!\left[\displaystyle\max_{1\leq i\leq N} \|\phi(u^{i,N})\|_2^{2p}\right]^{1/p} \!\!\E\!\left[\left|\int\! g\,d\mu \!-\! \int \!g\,d\widehat{\mu}^N\right|^{2q(1+\theta)}\right]^{1/q}}{(\int g\,d\mu)^2(1+\theta)} ,
\end{align*}
where we used H\"older inequality in the second inequality. Moreover, we have
\begin{align*}
\E\left[\max_{1\leq i\leq N} \|\phi(u^{i,N})\|_2^{2p}\right]^{1/p}&\leq \E\left[\sum_{i=1}^N \|\phi(u^{i,N})\|_2^{2p}\right]^{1/p}\\
&=N^{1/p} \left(\int \|\phi\|_2^{2p}\,d\mu\right)^{1/p}.
\end{align*}
Further from \cite[Equation (6.2)]{agapiou2017importance}, we have
\begin{gather*}
\E\!\left[\left|\int\! g\,d\mu \!-\! \int \!g\,d\widehat{\mu}^N\right|^{2q(1+\theta)}\!\right]^{1/q} \!\!\!\leq C_{2q(1+\theta)}^{1/q} \CM_{2q(1+\theta)}^{1/q} N^{-1-\theta}.
\end{gather*}
Picking $\theta=1/p\in(0,1)$, one has
\begin{gather}\label{A3K}
\sum_{k=1}^d A_{3,k}\leq \frac{C_{2q(1+\tfrac{1}{p})}^{1/q} \CM_{2q(1+\tfrac{1}{p})}^{1/q}\cdot \left(\int \|\phi\|_2^{2p}\,d\mu\right)^{1/p}\cdot N^{-1}}{(\int g\,d\mu)^2(1+\tfrac{1}{p})}.
\end{gather}
Combining \eqref{A1k}, \eqref{A2K} and \eqref{A3K} to finish the proof.
\end{proof}
\begin{proof}[Proof of Lemma~\ref{thm:importance sampling}]
From Theorem~\ref{thm:MSEbound} with $l=m=p=q=2$, $\mu = \rho_s$, $\phi(v) = v$ and $g(v) = e^{-\alpha \CE(v)}$, one has
\begin{gather}
\E\left[\left\|v_\alpha (\widehat{\overline{\rho}}^N_s) - v_\alpha (\rho_s)\right\|_2^2\right] \leq C_{\text{MSE}} N^{-1},
\end{gather}
where
\begin{align*}
C_{\text{MSE}} &= \frac{3\cdot\sum_{k=1}^{d} \left( \CM_{2}\left(v_k e^{-\alpha \CE(v)}\right)\right)}{(\int e^{-\alpha \CE(v)}\,d\rho_s)^2}\\
&+ \frac{27\cdot\left(\left(\int \|\phi\|_2^4 e^{-4\alpha \CE (v)}\,d\rho_s\right)^{1/2} \!\!\CM_4\left(e^{-\alpha \CE(v)}\right)^{1/2}\right)}{(\int e^{-\alpha \CE(v)}\,d\rho_s)^4}  \\
&+\frac{375\cdot\left(\left(\int \|v\|_2^{4}\,d\rho_s\right)^{1/2} \!\!\CM_{6}\left(e^{-\alpha \CE(v)}\right)^{1/2}\right)}{(\int e^{-\alpha \CE(v)}\,d\rho_s)^{3}}.
\end{align*}

\paragraph{\textbf{Lower bound of $\int e^{-\alpha \CE(v)}\,d\rho_s$}} Since $p=2\geq \max\{2,p_\CM(s,l)\}$ in our case, from Theorem~\ref{thm:mean-fild existence theorem}, we know $\E \left[\|\overline{V}_s\|_2^2\right] \leq C\cdot \left(\E \left[ \|\overline{V}_0\|_2^2 \right]+k_2(T)\|v^*\|_2^2\right)\cdot e^{C\cdot T\cdot k_2(T)} $. By Markov inequality, with
\begin{gather}\label{def of R}
R=\left(2C\cdot \left(\E \left[ \|\overline{V}_0\|_2^2 \right]+k_2(T)\|v^*\|_2^2\right)\cdot e^{C\cdot T\cdot k_2(T)}\right)^{1/2},
\end{gather} one has $\rho_s (B_R^c(0)) \leq \int \|v\|_2^2\,d\rho_s / R^2 \leq 1/2 $. Thus 
\begin{align*}
\int e^{-\alpha \CE(v)}\,d\rho_s &\geq \int_{B_R (0)} e^{-\alpha \CE(v)}\,d\rho_s\\ &\geq \int_{B_R (0)} e^{-\alpha(c_u R^l + C_u+\underline{\CE})}\,d\rho_s\\
&\geq \frac{1}{2}e^{-\alpha(c_u R^l + C_u + \underline{\CE})}.
\end{align*}
Then
\begin{equation}\label{weight_int_upper_bound}
\begin{split}
&\frac{1}{(\int e^{-\alpha \CE(v)}\,d\rho_s)^2} \vee \frac{1}{(\int e^{-\alpha \CE(v)}\,d\rho_s)^3} \vee \frac{1}{(\int e^{-\alpha \CE(v)}\,d\rho_s)^4}\\
&\leq 16 e^{4\alpha (c_u R^l + C_u + |\underline{\CE}|)},
\end{split}
\end{equation}
where $R$ is defined in~\eqref{def of R}.
\paragraph{\textbf{Upper bound of $\CM_{4}\left(e^{-\alpha \CE(v)}\right)^{1/2}$ and $\CM_{6}\left(e^{-\alpha \CE(v)}\right)^{1/2}$}} Since $\CE$ is lower bounded, we have
\begin{gather}\label{weight_moment_bound}
\CM_{4} (e^{-\alpha \CE(v)}) \vee \CM_{6} (e^{-\alpha \CE(v)}) \leq C.
\end{gather}
Putting \eqref{weight_int_upper_bound} and \eqref{weight_moment_bound} together, one obtains
\begin{equation}\label{whole_pre}
\begin{split}
C_{\text{MSE}}&\leq C e^{C(R^l + 1)}\Bigg(\Big(\sum_{k=1}^{d}  \CM_{2}\big(v_k e^{-\alpha \CE(v)}\big)\Big)+ \Big(\int \|v\|_2^4 e^{-4\alpha \CE (v)}\,d\rho_s\Big)^{1/2}+ \Big(\int \|v\|_2^{4}\,d\rho_s\Big)^{1/2}\Bigg),
\end{split}
\end{equation}
where $R$ is defined in~\eqref{def of R}.
\paragraph{\textbf{Upper bound of $\sum_{k=1}^{d}  \CM_{2}\left(v_k e^{-\alpha \CE(v)}\right)$}} Since $\CE$ is lower bounded, using Theorem~\ref{thm:mean-fild existence theorem}, one has
\begin{equation}
\begin{aligned}\label{bound_1}
\sum_{k=1}^{d}  \CM_{2}\left(v_k e^{-\alpha \CE(v)}\right)&\leq \sum_{k=1}^d \int v_k^2 e^{-2\alpha \CE(v)}\,d\rho_s\\
&\leq e^{-2\alpha \underline{\CE}} \int \|v\|_2^2 \,d\rho_s\\
&\leq e^{-2\alpha \underline{\CE}}\cdot C\cdot \left(\E \left[ \|\overline{V}_0\|_2^2 \right]+k_2(T)\|v^*\|_2^2\right)\cdot e^{C\cdot T\cdot k_2(T)}.
\end{aligned}
\end{equation}
\paragraph{\textbf{Upper bounds of $ \left(\int \|v\|_2^4 e^{-4\alpha \CE (v)}\,d\rho_s\right)^{1/2}$ and $\left(\int \|v\|_2^{4}\,d\rho_s\right)^{1/2}$}} Since $\CE$ is lower bounded, it suffices to bound $ \left(\int \|v\|_2^{4}\,d\rho_s\right)^{1/2}$. By Theorem~\ref{thm:mean-fild existence theorem}, one has
\begin{align*}
\left(\int \|v\|_2^{4}\,d\rho_s\right)^{1/2}\leq C\cdot \left(\E \left[ \|\overline{V}_0\|_2^4 \right]+k_4(T)\|v^*\|_2^4\right)^{1/2}\cdot e^{C\cdot T\cdot k_4(T)}.
\end{align*}
Thus, one has
\begin{align*}
C_{\text{MSE}}\leq C e^{\Phi_{\rho_0, T,v^*}}\Big(\max\big\{k_4(T),k_2^2(T)\big\}\|v^*\|_2^4 + \E[\|\overline{V}_0\|_2^4]\Big)^{1/2},
\end{align*}
where 
\begin{align*}
\Phi_{\rho_0,T,v^*} = C\Bigg(1+T\cdot \max\big\{k_2(T),k_4(T)\big\}+ \Big(\E \left[ \|\overline{V}_0\|_2^2 \right]+k_2(T)\|v^*\|_2^2\Big)^{l/2}\cdot e^{C\cdot T\cdot k_2(T)}\Bigg).
\end{align*}
This completes the proof.
\end{proof}

\section{\HZnew{Naive Euler--Maruyama Scheme}}\label{EM}

\HZnew{In this section, we provide a direct Euler--Maruyama scheme of~\eqref{dyn} and establish its finite-time convergence. The Euler--Maruyama discretization of~\eqref{dyn} is summarized in Algorithm~\ref{algo:EM}. We then have the following convergence result.}
\HZnew{\begin{algorithm}[t]
\caption{Euler--Maruyama discretization of \textit{ProxiCBO}}
\label{algo:EM}
\begin{algorithmic}[1]
\State \textbf{Input:} $\{V^i\}_{i=1}^N\overset{\mathrm{i.i.d.}}{\sim}\rho_0$, 
$D(\cdot)=\operatorname{diag}(\cdot)$
\While{not STOP}
    \State \textbf{Compute consensus point:}
    \[
    v_\alpha \gets
    \frac{\sum_{i=1}^{N}\exp\!\left(-\alpha\,\CE(V^i)\right)V^i}
    {\sum_{i=1}^{N}\exp\!\left(-\alpha\,\CE(V^i)\right)}.
    \]
    \State \textbf{Update particles:} for $i=1,\ldots,N$
    \begin{align*}
    V^i \gets\;&V^i-\lambda_1\left(V^i-v_\alpha\right)\Delta t\\
    &-\lambda_2
    \left(
    \nabla f(V^i)
    +\nabla M_{\mu g}\left(V^i-\mu\nabla f(V^i)\right)
    \right)\Delta t\\
    &+\sigma_1D\left(V^i-v_\alpha\right)z^{i,1}\sqrt{\Delta t}\\
    &+\sigma_2D\Big(
    \nabla f(V^i)+\nabla M_{\mu g}\left(V^i-\mu\nabla f(V^i)\Big)
    \right)z^{i,2}\sqrt{\Delta t},
    \end{align*}
    where $\{z^{i,1}\}_{i=1}^N$ and $\{z^{i,2}\}_{i=1}^N$ are independent
    $d$-dimensional standard Gaussian vectors.
\EndWhile
\State \textbf{Output:} Historical or current best particle location
\end{algorithmic}
\end{algorithm}}
\begin{theorem}\label{thm:EM}
\HZnew{Let $N\in\N$ and $\mu>0$ be fixed, and suppose Assumption~\ref{assump} (ii) and (iii) hold. Assume $V^i_0\overset{\mathrm{i.i.d.}}{\sim}\rho_0$ with $\rho_0\in\mathcal{P}_p(\R^d)$ for some $p>2$. Let
\begin{gather*}
\mathbf{V}_t
=
\left(\left(V_t^{1,N}\right)^\top,\ldots,
\left(V_t^{N,N}\right)^\top\right)^\top
\end{gather*}
be the solution to~\eqref{dyn}. Using the notation in the proof of Theorem~\ref{thm:wellposed for micro} (for simplicity, we only present the result for $\lambda_1=\lambda_2=\sigma_1=\sigma_2=1$. The general case follows identically), for $t\in[n\Delta t,(n+1)\Delta t)$, define
\begin{align*}
\mathbf{V}^{\Delta t}_t =\mathbf{V}_n +(t-n\Delta t)
\left(F_N+\widetilde{F}_N\right)(\mathbf{V}_n)+M_N(\mathbf{V}_n)
\left(B_t^{(N)}-B_{n\Delta t}^{(N)}\right),
\end{align*}
where $\mathbf{V}_n$ denotes the iterates generated by Algorithm~\ref{algo:EM}. Then for every $T>0$, $
\lim_{\Delta t\rightarrow0}
\E\left[
\sup_{t\in[0,T]}
\left\|\mathbf{V}_t-\mathbf{V}^{\Delta t}_t\right\|_2^2
\right]
=0.$}
\end{theorem}

\begin{proof}
\HZnew{We verify the conditions of~\cite[Theorem~2.2]{higham2002strong}.
As in the proof of Theorem~\ref{thm:wellposed for micro}, we concatenate the particles and write~\eqref{dyn} as
\begin{gather*}
d\mathbf{V}_t
=
\left(F_N(\mathbf{V}_t)+\widetilde{F}_N(\mathbf{V}_t)\right)dt
+
M_N(\mathbf{V}_t)dB_t^{(N)}.
\end{gather*}
As shown in the proof of Theorem~\ref{thm:wellposed for micro},
$F_N+\widetilde{F}_N$ and $M_N$ are locally Lipschitz and satisfy
\begin{gather}\label{eq:EM-growth}
\left\|F_N(\mathbf{V})+\widetilde{F}_N(\mathbf{V})\right\|_2^2
+
\left\|M_N(\mathbf{V})\right\|_F^2
\leq C\left(1+\|\mathbf{V}\|_2^2\right).
\end{gather}}

\HZnew{It remains to verify the moment bounds required in
\cite[Assumption~2.1]{higham2002strong}. 
Let $\overline{\mathbf{V}}^{\Delta t}_t=\mathbf{V}_n$ for
$t\in[n\Delta t,(n+1)\Delta t)$. By the definition of the continuous
interpolation,
\begin{align*}
\mathbf{V}^{\Delta t}_t
=\mathbf{V}_0
+\int_0^t
\left(F_N+\widetilde{F}_N\right)
\left(\overline{\mathbf{V}}^{\Delta t}_s\right)ds+\int_0^t
M_N\left(\overline{\mathbf{V}}^{\Delta t}_s\right)dB_s^{(N)}.
\end{align*}
Then, using~\eqref{eq:EM-growth}, H\"older's inequality, and the
Burkholder--Davis--Gundy inequality, for $p>2$ we have
\begin{align*}
\E\left[
\sup_{r\in[0,t]}
\|\mathbf{V}^{\Delta t}_r\|_2^p
\right]
&\leq
C\left(1+\E\|\mathbf{V}_0\|_2^p\right)+
C\int_0^t
\E\left[
\sup_{r\in[0,s]}
\|\mathbf{V}^{\Delta t}_r\|_2^p
\right]ds.
\end{align*}
Thus, one can use Gr\"onwall's inequality to obtain $
\sup_{0<\Delta t\leq 1}
\E\left[
\sup_{t\in[0,T]}
\|\mathbf{V}^{\Delta t}_t\|_2^p
\right]<\infty,$ where $C$ is independent of $\Delta t$. The same estimate applies
to the continuous-time system. Hence Assumption~2.1 in~\cite{higham2002strong} is satisfied. The result follows directly from~\cite[Theorem~2.2]{higham2002strong}.}
\end{proof}
\bibliographystyle{abbrv}
\bibliography{citations}
\end{document}